\documentclass{article}
\usepackage{amsfonts}
\usepackage{amsmath}

\newtheorem{theorem}{Theorem}

\newtheorem{definition}[theorem]{Definition}

\newtheorem{notation}[theorem]{Notation}

\newtheorem{remark}[theorem]{Remark}

\newenvironment{proof}[1][Proof]{\noindent\textbf{#1.} }{\ \rule{0.5em}{0.5em}}
\input{tcilatex}
\begin{document}

\title{Homogenization of linear parabolic equations with three spatial and
three temporal scales for certain matchings between the microscopic scales}
\author{Tatiana Danielsson and Pernilla Johnsen }
\maketitle

\begin{abstract}
In this paper we establish compactness results of multiscale and very weak
multiscale type for sequences bounded in $L^{2}(0,T;H_{0}^{1}(\Omega ))$,
fulfilling a certain condition. We apply the results in the homogenization
of $\varepsilon ^{p}\partial _{t}u_{\varepsilon }\left( x,t\right) -\nabla
\cdot \left( a\left( x/\varepsilon ,x/\varepsilon ^{2},t/\varepsilon
^{q},t/\varepsilon ^{r}\right) \nabla u_{\varepsilon }\left( x,t\right)
\right) =f\left( x,t\right) $, where $0<p<q<r$. The homogenization result
reveals two special phenomena, namely that the homogenized problem is
elliptic and that the matching for when the local problem is parabolic is
shifted by $p$, compared to the standard matching that gives rise to local
parabolic problems.
\end{abstract}

\section{Introduction}

Let $T>0$ and $\Omega _{T}=\Omega \times \left( 0,T\right) $, where $\Omega $
is an open bounded subset of $%
\mathbb{R}
^{N}$ with smooth boundary and $(0,T)$ is an open bounded interval in $%
\mathbb{R}
$. We consider the homogenization of the linear parabolic equation%
\begin{eqnarray}
\varepsilon ^{p}\partial _{t}u_{\varepsilon }\left( x,t\right) -\nabla \cdot
\left( a\left( \frac{x}{\varepsilon },\frac{x}{\varepsilon ^{2}},\frac{t}{%
\varepsilon ^{q}},\frac{t}{\varepsilon ^{r}}\right) \nabla u_{\varepsilon
}\left( x,t\right) \right) &=&f\left( x,t\right) \text{ in }\Omega _{T}\text{%
,}  \notag \\
u_{\varepsilon }\left( x,0\right) &=&u_{0}\left( x\right) \text{ in }\Omega 
\text{,}  \label{problem} \\
u_{\varepsilon }\left( x,t\right) &=&0\text{ on }\partial \Omega \times
\left( 0,T\right) \text{,}  \notag
\end{eqnarray}%
where $0<p<q<r$ are real numbers, $f\in L^{2}(\Omega _{T})$ and $u_{0}\in
L^{2}(\Omega )$. The coefficient $a$ is periodic with respect to the unit
cube $Y=(0,1)^{N}$ in the first two variables and with respect to the unit
interval $S=(0,1)$ in the third and fourth variable. More detailed
information on the equation will be provided in Section \ref{Homogenization}.

Homogenization means that we study the limit behavior as $\varepsilon
\rightarrow 0$ and search for a weak $L^{2}(0,T;H_{0}^{1}(\Omega ))$-limit $%
u $ to $\left\{ u_{\varepsilon }\right\} $ which is the solution to a
so-called homogenized problem. This limit problem is governed by a
coefficient $b$ that unlike $a\left( x/\varepsilon ,x/\varepsilon
^{2},t/\varepsilon ^{q},t/\varepsilon ^{r}\right) $ does not include rapid
oscillations. In the homogenization procedure local problems are also
extracted which include information about the microstructure and whose
solutions are utilized to determine $b$.

The present paper is a further generalization of the work presented in \cite%
{JoLo1}. In earlier works, like e.g. \cite{FHOLParbitrary}, boundedness in $%
W^{1,2}(0,T;H_{0}^{1}(\Omega ),L^{2}(\Omega ))$, meaning that $\left\{
u_{\varepsilon }\right\} $ is bounded in $L^{2}(0,T;H_{0}^{1}(\Omega ))$ and 
$\left\{ \partial _{t}u_{\varepsilon }\right\} $ is bounded in $%
L^{2}(0,T;H^{-1}(\Omega ))$, has been required when compactness results have
been established. In \cite{JoLo1}, compactness results of $\left( 2,2\right) 
$-scale and very weak $\left( 2,2\right) $-scale convergence type were
proven by requiring boundedness of the sequence $\left\{ u_{\varepsilon
}\right\} $ in $L^{2}(0,T;H_{0}^{1}(\Omega ))$ but replacing the assumption
of boundedness of the time derivative in $L^{2}(0,T;H^{-1}(\Omega ))$ by a
certain condition. This new approach originates, up to the authors'
knowledge, from \cite{Lob} and will be used in the present work. Here we
focus on establishing appropriate compactness results and a homogenization
result for the parabolic partial differential equation (\ref{problem}). In
particular, we generalize the result from \cite{JoLo1} to the $(2,3)$-scale
and $\left( 3,3\right) $-scale convergence types, adapting to the problem (%
\ref{problem}), and present compactness results for both multiscale and very
weak multiscale convergence.

For the homogenization part of this paper we apply the convergence results
to establish a homogenization result for (\ref{problem}) with 13 different
outcomes, depending on the choices of parameters $p$, $q$ and $r$. The
homogenization result will reveal two phenomena, which also occurred in both 
\cite{JoLo1} and the proceeding work \cite{Jolo2}, where the homogenization
of parabolic problems of a similar kind, but with only one rapid scale in
space and time each, was presented. The first phenomenon is that the
homogenized problem is of elliptic type even though the original problem is
a parabolic one and the second is that resonance occurs for different
matchings between the microscopic scales than the standard ones. By
resonance we mean that the local problem is parabolic, which only occurs for
certain matchings between the microscopic scales. What we call the standard
matching is when a temporal scale equals the square of a spatial one, as was
the case in several other studies, see e.g. \cite{BLP}, \cite{Hol}, \cite%
{NgWo}, \cite{AlPi}, \cite{FHOLstrange-term}, \cite{FHOLPmismatch}, \cite%
{SvWo}, \cite{FHOLParbitrary} or \cite{DoWo} for more on this matter.
However, in our case the matching for when we have resonance is shifted by $%
p $. Note that in our equation, (\ref{problem}), we would get resonance for
the standard matching if $p=0$, cf. Section 5.3.1 in \cite{PerPhD}.

The paper is organized as follows. In Section 2 we recall some of the key
definitions, namely evolution multiscale convergence and very weak evolution
multiscale convergence. We prove the main convergence results (see Theorems %
\ref{Theorem - gradient charact} and \ref{Theorem - Compactresult vw}),
which lay the foundation to establish the homogenization result. Theorem \ref%
{Theorem - gradient charact} is where we find characterizations of the $%
(2,3) $-scale and $\left( 3,3\right) $-scale limits for $\left\{ \nabla
u_{\varepsilon }\right\} $ under certain assumptions. In Theorem \ref%
{Theorem - Compactresult vw} we consider very weak $\left( 2,3\right) $%
-scale and $\left( 3,3\right) $-scale convergence for the sequences $\left\{
\varepsilon ^{-1}u_{\varepsilon }\right\} $ and $\left\{ \varepsilon
^{-2}u_{\varepsilon }\right\} $, respectively. In Section 3, we state a
homogenization result presented in Theorem \ref{Theorem - homogenisering}.

We end the introduction with some essential notations used throughout this
paper.

\begin{notation}
\label{notations}We denote $\mathcal{Y}_{n,m}=Y^{n}\times S^{m}$ with $%
Y^{n}=Y_{1}\times Y_{2}\times \cdots \times Y_{n}$ and $S^{m}=S_{1}\times
S_{2}\times \cdots \times S_{m}$, where $Y_{1}=Y_{2}=\ldots =Y_{n}=Y=\left(
0,1\right) ^{N}$ and $S_{1}=S_{2}=\ldots =S_{m}=S=\left( 0,1\right) $. We
let $y^{n}=y_{1},y_{2},\ldots ,y_{n}$, $dy^{n}=dy_{1}dy_{2}\cdots dy_{n}$, $%
s^{m}=s_{1},s_{2},\ldots ,s_{m}$ and $ds^{m}=ds_{1}ds_{2}\cdots ds_{m}$. We
define the function space $\mathcal{W}_{i,j}=\left\{ u\in L_{\sharp
}^{2}(S_{j};H_{\sharp }^{1}(Y_{i})/%
\mathbb{R}
):\partial _{s_{j}}u\in L_{\sharp }^{2}(S_{j};(H_{\sharp }^{1}(Y_{i})/%
\mathbb{R}
)^{\prime })\right\} $. The subscript $_{\sharp }$ is used to denote
periodicity of the functions involved over the domain in question. Lastly,
for $k=1,\ldots ,n$ and $j=1,\ldots ,m$, the scale functions $\varepsilon
_{k}\left( \varepsilon \right) $ and $\varepsilon _{j}^{\prime }\left(
\varepsilon \right) $ are strictly positive functions that tend to zero as $%
\varepsilon $ does and $\left\{ \varepsilon _{1},\ldots ,\varepsilon
_{n}\right\} $ and $\left\{ \varepsilon _{1}^{\prime },\ldots ,\varepsilon
_{m}^{\prime }\right\} $ denote lists of spatial and temporal scales,
respectively.
\end{notation}

\section{Multiscale and very weak multiscale convergence}

The concept of multiscale convergence is a generalization of the classical
two-scale convergence, originating from \cite{Ngu1} and \cite{Ngu2}.
Two-scale convergence is suitable for sequences having one microscopic
spatial scale and it has been generalized, first to include multiple spatial
scales by Allaire and Briane in \cite{AlBr}, and later to also include
multiple temporal scales.

\begin{definition}
\label{Definition 2} A sequence $\left\{ u_{\varepsilon }\right\} $ in $%
L^{2}(\Omega _{T})$ is said to $\left( n+1,m+1\right) $-scale converge to $%
u_{0}\in L^{2}(\Omega _{T}\times \mathcal{Y}_{n,m})$ if%
\begin{gather*}
\lim_{\varepsilon \rightarrow 0}\int_{\Omega _{T}}u_{\varepsilon }\left(
x,t\right) v\left( x,t,\frac{x}{\varepsilon _{1}},\cdots ,\frac{x}{%
\varepsilon _{n}},\frac{t}{\varepsilon _{1}^{\prime }},\cdots ,\frac{t}{%
\varepsilon _{m}^{\prime }}\right) dxdt \\
=\int_{\Omega _{T}}\int_{\mathcal{Y}_{n,m}}u_{0}\left(
x,t,y^{n},s^{m}\right) v\left( x,t,y^{n},s^{m}\right) dy^{n}ds^{m}dxdt
\end{gather*}%
for all $v\in L^{2}(\Omega _{T};C_{\sharp }(\mathcal{Y}_{n,m})).$ This is
denoted by 
\begin{equation*}
u_{\varepsilon }\left( x,t\right) \overset{n+1,m+1}{\rightharpoonup }%
u_{0}\left( x,t,y^{n},s^{m}\right) \text{.}
\end{equation*}
\end{definition}

We make some standard assumptions on the scales. We say that the scales in a
list $\left\{ \varepsilon _{1},\ldots ,\varepsilon _{n}\right\} $ are
separated if%
\begin{equation*}
\lim_{\varepsilon \rightarrow 0}\frac{\varepsilon _{k+1}}{\varepsilon _{k}}=0
\end{equation*}%
and well-separated if there exists a positive integer $\ell $ such that%
\begin{equation*}
\lim_{\varepsilon \rightarrow 0}\frac{1}{\varepsilon _{k}}\left( \frac{%
\varepsilon _{k+1}}{\varepsilon _{k}}\right) ^{\ell }=0\text{,}
\end{equation*}%
where $k=1,\ldots ,n-1$. Following the definition by Persson, see e.g. \cite%
{Per}, the generalization of separatedness and well-separatedness to include
two lists of scales reads as follows.

\begin{definition}
Let $\left\{ \varepsilon _{1},\ldots ,\varepsilon _{n}\right\} $ and $%
\left\{ \varepsilon _{1}^{\prime },\ldots ,\varepsilon _{m}^{\prime
}\right\} $ be lists of (well-)separated scales. Collect all elements from
both lists in one common list. If from possible duplicates, where by
duplicates we mean scales which tend to zero equally fast, one member of
each pair is removed and the list in order of magnitude of all the remaining
elements is (well-)separated, the lists $\left\{ \varepsilon _{1},\ldots
,\varepsilon _{n}\right\} $ and $\left\{ \varepsilon _{1}^{\prime },\ldots
,\varepsilon _{m}^{\prime }\right\} $ are said to be jointly
(well-)separated.
\end{definition}

We present a compactness result for evolution multiscale convergence.

\begin{theorem}
\label{Theorem - begr ger 3,3}Let $\left\{ u_{\varepsilon }\right\} $ be a
bounded sequence in $L^{2}(\Omega _{T})$ and suppose that the lists $\left\{
\varepsilon _{1},\ldots ,\varepsilon _{n}\right\} $ and $\left\{ \varepsilon
_{1}^{\prime },\ldots ,\varepsilon _{m}^{\prime }\right\} $ are jointly
separated. Then, up to a subsequence,%
\begin{equation*}
u_{\varepsilon }\left( x,t\right) \overset{n+1,m+1}{\rightharpoonup }%
u_{0}\left( x,t,y^{n},s^{m}\right) \text{,}
\end{equation*}%
where $u_{0}\in L^{2}(\Omega _{T}\times \mathcal{Y}_{n,m})$.
\end{theorem}

\begin{proof}
See Theorem A.1 in \cite{FHOLParbitrary}.
\end{proof}

As the next theorem states, the evolution multiscale limit is unique.

\begin{theorem}
\label{Theorem - multiskala unik}The $\left( n+1,m+1\right) $-scale limit is
unique.
\end{theorem}

\begin{proof}
The proof is analogous to the proof of the uniqueness of the two-scale limit
given in the discussion below Definition 1 in \cite{LNW}.
\end{proof}

We are now ready to give a compactness result for the gradient of a sequence 
$\left\{ u_{\varepsilon }\right\} $. The following theorem will play a vital
role in the homogenization of (\ref{problem}).

\begin{theorem}
\label{Theorem - gradient charact}Let $\left\{ u_{\varepsilon }\right\} $ be
a bounded sequence in $L^{2}(0,T;H_{0}^{1}(\Omega ))$ and, for any $v\in
D(\Omega )$, $c_{1}\in D(0,T)$, $c_{2}\in C_{\sharp }^{\infty }(S_{1})$, $%
c_{3}\in C_{\sharp }^{\infty }(S_{2})$ and $r>q>0$,%
\begin{equation}
\lim_{\varepsilon \rightarrow 0}\int_{\Omega _{T}}u_{\varepsilon }\left(
x,t\right) v\left( x\right) \partial _{t}\left( \varepsilon ^{r}c_{1}\left(
t\right) c_{2}\left( \frac{t}{\varepsilon ^{q}}\right) c_{3}\left( \frac{t}{%
\varepsilon ^{r}}\right) \right) dxdt=0  \label{villkor1}
\end{equation}%
and%
\begin{equation}
\lim_{\varepsilon \rightarrow 0}\int_{\Omega _{T}}u_{\varepsilon }\left(
x,t\right) v\left( x\right) \partial _{t}\left( \varepsilon ^{q}c_{1}\left(
t\right) c_{2}\left( \frac{t}{\varepsilon ^{q}}\right) \right) dxdt=0\text{.}
\label{villkor2}
\end{equation}%
Then, with $\varepsilon _{1}=\varepsilon $, $\varepsilon _{2}=\varepsilon
^{2}$, $\varepsilon _{1}^{\prime }=\varepsilon ^{q}$ and $\varepsilon
_{2}^{\prime }=\varepsilon ^{r}$, up to a subsequence,%
\begin{equation}
u_{\varepsilon }\left( x,t\right) \rightharpoonup u\left( x,t\right) \text{
in }L^{2}(0,T;H_{0}^{1}(\Omega ))\text{,}  \label{svag konvergens}
\end{equation}%
\begin{equation}
u_{\varepsilon }\left( x,t\right) \overset{3,3}{\rightharpoonup }u\left(
x,t\right) \text{,}  \label{3,3}
\end{equation}%
\begin{equation}
\nabla u_{\varepsilon }\left( x,t\right) \overset{2,3}{\rightharpoonup }%
\nabla u\left( x,t\right) +\nabla _{y_{1}}u_{1}\left( x,t,y_{1},s^{2}\right)
\label{gradsplit1}
\end{equation}%
and%
\begin{equation}
\nabla u_{\varepsilon }\left( x,t\right) \overset{3,3}{\rightharpoonup }%
\nabla u\left( x,t\right) +\nabla _{y_{1}}u_{1}\left( x,t,y_{1},s^{2}\right)
+\nabla _{y_{2}}u_{2}\left( x,t,y^{2},s^{2}\right) \text{,}
\label{gradsplit2}
\end{equation}%
where $u\in L^{2}(0,T;H_{0}^{1}(\Omega ))$, $u_{1}\in L^{2}(\Omega
_{T}\times S^{2};H_{\sharp }^{1}(Y_{1})/%
\mathbb{R}
)$ and $u_{2}\in L^{2}(\Omega _{T}\times \mathcal{Y}_{1,2};H_{\sharp
}^{1}(Y_{2})/%
\mathbb{R}
)$.
\end{theorem}

\begin{proof}
From the boundedness of $\left\{ u_{\varepsilon }\right\} $ in $%
L^{2}(0,T;H_{0}^{1}(\Omega ))$, the weak convergence (\ref{svag konvergens})
follows immediately. It also implies that $\left\{ \nabla u_{\varepsilon
}\right\} $ is bounded in $L^{2}(\Omega _{T})^{N}$ and hence, according to
Theorems \ref{Theorem - begr ger 3,3} and \ref{Theorem - multiskala unik},
we have%
\begin{equation}
u_{\varepsilon }\left( x,t\right) \overset{3,3}{\rightharpoonup }u_{0}\left(
x,t,y^{2},s^{2}\right)  \label{u0}
\end{equation}%
and 
\begin{equation}
\nabla u_{\varepsilon }\left( x,t\right) \overset{3,3}{\rightharpoonup }\tau
_{0}\left( x,t,y^{2},s^{2}\right) \text{,}  \label{tau}
\end{equation}%
up to a subsequence, for some unique $u_{0}\in L^{2}(\Omega _{T}\times 
\mathcal{Y}_{2,2})$ and $\tau _{0}\in L^{2}(\Omega _{T}\times \mathcal{Y}%
_{2,2})^{N}$.

We proceed by characterizing $u_{0}$, where we first show that $u_{0}$ is
independent of the local space and time variables $y_{1}$, $y_{2}$, $s_{1}$
and $s_{2}$. Letting $v_{1}\in \emph{D}(\Omega )$, $v_{2}\in C_{\sharp
}^{\infty }(Y_{1})$, $v_{3}\in C_{\sharp }^{\infty }(Y_{2})^{N}$, $c_{1}\in 
\emph{D}(0,T)$, $c_{2}\in C_{\sharp }^{\infty }(S_{1})$ and $c_{3}\in
C_{\sharp }^{\infty }(S_{2})$, it holds that%
\begin{gather*}
\int_{\Omega _{T}}\nabla u_{\varepsilon }\left( x,t\right) \varepsilon
^{2}v_{1}\left( x\right) v_{2}\left( \frac{x}{\varepsilon }\right) \cdot
v_{3}\left( \frac{x}{\varepsilon ^{2}}\right) c_{1}\left( t\right)
c_{2}\left( \frac{t}{\varepsilon ^{q}}\right) c_{3}\left( \frac{t}{%
\varepsilon ^{r}}\right) dxdt \\
=-\int_{\Omega _{T}}u_{\varepsilon }\left( x,t\right) \left( \varepsilon
^{2}\nabla v_{1}\left( x\right) v_{2}\left( \frac{x}{\varepsilon }\right)
\cdot v_{3}\left( \frac{x}{\varepsilon ^{2}}\right) +\varepsilon v_{1}\left(
x\right) \nabla _{y_{1}}v_{2}\left( \frac{x}{\varepsilon }\right) \cdot
v_{3}\left( \frac{x}{\varepsilon ^{2}}\right) \right. \\
+\left. v_{1}\left( x\right) v_{2}\left( \frac{x}{\varepsilon }\right)
\nabla _{y_{2}}\cdot v_{3}\left( \frac{x}{\varepsilon ^{2}}\right) \right)
c_{1}\left( t\right) c_{2}\left( \frac{t}{\varepsilon ^{q}}\right)
c_{3}\left( \frac{t}{\varepsilon ^{r}}\right) dxdt\text{,}
\end{gather*}%
where we have applied integration by parts and carried out the
differentiation process. As $\varepsilon \rightarrow 0$, $\left\{
\varepsilon ^{2}\nabla u_{\varepsilon }\right\} $ approaches $0$ due to
boundedness of $\left\{ \nabla u_{\varepsilon }\right\} $ and we obtain%
\begin{gather*}
\lim_{\varepsilon \rightarrow 0}\int_{\Omega _{T}}-u_{\varepsilon }\left(
x,t\right) \left( \varepsilon ^{2}\nabla v_{1}\left( x\right) v_{2}\left( 
\frac{x}{\varepsilon }\right) \cdot v_{3}\left( \frac{x}{\varepsilon ^{2}}%
\right) +\varepsilon v_{1}\left( x\right) \nabla _{y_{1}}v_{2}\left( \frac{x%
}{\varepsilon }\right) \cdot v_{3}\left( \frac{x}{\varepsilon ^{2}}\right)
\right. \\
+\left. v_{1}\left( x\right) v_{2}\left( \frac{x}{\varepsilon }\right)
\nabla _{y_{2}}\cdot v_{3}\left( \frac{x}{\varepsilon ^{2}}\right) \right)
c_{1}\left( t\right) c_{2}\left( \frac{t}{\varepsilon ^{q}}\right)
c_{3}\left( \frac{t}{\varepsilon ^{r}}\right) dxdt=0
\end{gather*}%
and since all but the third term vanish, (\ref{u0}) gives%
\begin{gather*}
\int_{\Omega _{T}}\int_{\mathcal{Y}_{2,2}}-u_{0}\left(
x,t,y^{2},s^{2}\right) v_{1}\left( x\right) v_{2}\left( y_{1}\right) \nabla
_{y_{2}}\cdot v_{3}\left( y_{2}\right) \\
\times c_{1}\left( t\right) c_{2}\left( s_{1}\right) c_{3}\left(
s_{2}\right) dy^{2}ds^{2}dxdt=0\text{.}
\end{gather*}%
Applying the Variational Lemma we have%
\begin{equation*}
-\int_{Y_{2}}u_{0}\left( x,t,y^{2},s^{2}\right) \nabla _{y_{2}}\cdot
v_{3}\left( y_{2}\right) \emph{d}y_{2}=0
\end{equation*}%
a.e. in $\Omega _{T}\times \mathcal{Y}_{1,2}$, showing that $u_{0}$ is
independent of $y_{2}$. Next we let $v_{1}\in \emph{D}(\Omega )$, $v_{2}\in
C_{\sharp }^{\infty }(Y_{1})^{N}$, $c_{1}\in \emph{D}(0,T),\ c_{2}\in
C_{\sharp }^{\infty }(S_{1})$ and $c_{3}\in C_{\sharp }^{\infty }(S_{2}).$
By integration by parts and after differentiation we have that%
\begin{gather*}
\int_{\Omega _{T}}\nabla u_{\varepsilon }\left( x,t\right) \varepsilon
v_{1}\left( x\right) \cdot v_{2}\left( \frac{x}{\varepsilon }\right)
c_{1}\left( t\right) c_{2}\left( \frac{t}{\varepsilon ^{q}}\right)
c_{3}\left( \frac{t}{\varepsilon ^{r}}\right) \emph{d}x\emph{d}t \\
=-\int_{\Omega _{T}}u_{\varepsilon }\left( x,t\right) \left( \varepsilon
\nabla v_{1}\left( x\right) \cdot v_{2}\left( \frac{x}{\varepsilon }\right)
+v_{1}\left( x\right) \nabla _{y_{1}}\cdot v_{2}\left( \frac{x}{\varepsilon }%
\right) \right) \\
\times c_{1}\left( t\right) c_{2}\left( \frac{t}{\varepsilon ^{q}}\right)
c_{3}\left( \frac{t}{\varepsilon ^{r}}\right) \emph{d}x\emph{d}t
\end{gather*}%
and as $\varepsilon \rightarrow 0$ we obtain%
\begin{equation*}
\int_{\Omega _{T}}\int_{\mathcal{Y}_{1,2}}-u_{0}\left(
x,t,y_{1},s^{2}\right) v_{1}\left( x\right) \nabla _{y_{1}}\cdot v_{2}\left(
y_{1}\right) c_{1}\left( t\right) c_{2}\left( s_{1}\right) c_{3}\left(
s_{2}\right) \emph{d}y_{1}\emph{d}s^{2}\emph{d}x\emph{d}t=0\text{.}
\end{equation*}%
By the Variational Lemma%
\begin{equation*}
-\int_{Y_{1}}u_{0}\left( x,t,y_{1},s^{2}\right) \nabla _{y_{1}}\cdot
v_{2}\left( y_{1}\right) \emph{d}y_{1}=0
\end{equation*}%
a.e. in $\Omega _{T}\times S^{2}$, which shows that $u_{0}$ is independent
of $y_{1}$. To show independence of $s_{2}$ we carry out the
differentiations in (\ref{villkor1}) and obtain%
\begin{gather*}
\lim_{\varepsilon \rightarrow 0}\int_{\Omega _{T}}u_{\varepsilon }\left(
x,t\right) v\left( x\right) \left( \varepsilon ^{r}\partial _{t}c_{1}\left(
t\right) c_{2}\left( \frac{t}{\varepsilon ^{q}}\right) c_{3}\left( \frac{t}{%
\varepsilon ^{r}}\right) \right. \\
+\left. \varepsilon ^{r-q}c_{1}\left( t\right) \partial _{s_{1}}c_{2}\left( 
\frac{t}{\varepsilon ^{q}}\right) c_{3}\left( \frac{t}{\varepsilon ^{r}}%
\right) +\varepsilon ^{r-r}c_{1}\left( t\right) c_{2}\left( \frac{t}{%
\varepsilon ^{q}}\right) \partial _{s_{2}}c_{3}\left( \frac{t}{\varepsilon
^{r}}\right) \right) \emph{d}x\emph{d}t=0\text{.}
\end{gather*}%
Passing to the limit we arrive at%
\begin{equation*}
\int_{\Omega _{T}}\int_{S^{2}}u_{0}\left( x,t,s^{2}\right) v\left( x\right)
c_{1}\left( t\right) c_{2}\left( s_{1}\right) \partial _{s_{2}}c_{3}\left(
s_{2}\right) \emph{d}s^{2}\emph{d}x\emph{d}t=0
\end{equation*}%
and the Variational Lemma gives%
\begin{equation*}
\int_{S_{2}}u_{0}\left( x,t,s^{2}\right) \partial _{s_{2}}c_{3}\left(
s_{2}\right) \emph{d}s_{2}=0
\end{equation*}%
a.e. in $\Omega _{T}\times S_{1}$. We conclude that $u_{0}$ does not depend
on the local time variable $s_{2}$. For showing independence of $s_{1}$ we
carry out the differentiations in (\ref{villkor2}) and obtain%
\begin{equation*}
\lim_{\varepsilon \rightarrow 0}\int_{\Omega _{T}}u_{\varepsilon }\left(
x,t\right) v\left( x\right) \left( \varepsilon ^{q}\partial _{t}c_{1}\left(
t\right) c_{2}\left( \frac{t}{\varepsilon ^{q}}\right) +\varepsilon
^{q-q}c_{1}\left( t\right) \partial _{s_{1}}c_{2}\left( \frac{t}{\varepsilon
^{q}}\right) \right) \emph{d}x\emph{d}t=0\text{.}
\end{equation*}%
As $\varepsilon $ tends to zero we have%
\begin{equation*}
\int_{\Omega _{T}}\int_{S_{1}}u_{0}\left( x,t,s_{1}\right) v\left( x\right)
c_{1}\left( t\right) \partial _{s_{1}}c_{2}\left( s_{1}\right) \emph{d}s_{1}%
\emph{d}x\emph{d}t=0
\end{equation*}%
and by the Variational Lemma%
\begin{equation*}
\int_{S_{1}}u_{0}\left( x,t,s_{1}\right) \partial _{s_{1}}c_{2}\left(
s_{1}\right) \emph{d}s_{1}=0
\end{equation*}%
a.e. in $\Omega _{T}$, hence $u_{0}$ is independent of $s_{1}$. In
conclusion, we have shown that%
\begin{equation}
u_{\varepsilon }\left( x,t\right) \overset{3,3}{\rightharpoonup }u_{0}\left(
x,t\right) \text{,}  \label{u0-2}
\end{equation}%
where $u_{0}\in L^{2}(\Omega _{T})$, and the last step in the
characterization of $u_{0}$ is to show that $u_{0}\in
L^{2}(0,T;H_{0}^{1}(\Omega ))$. Observe that (\ref{u0-2}) means%
\begin{gather*}
\lim_{\varepsilon \rightarrow 0}\int_{\Omega _{T}}u_{\varepsilon }\left(
x,t\right) v\left( x,t,\frac{x}{\varepsilon },\frac{x}{\varepsilon ^{2}},%
\frac{t}{\varepsilon ^{q}},\frac{t}{\varepsilon ^{r}}\right) \emph{d}x\emph{d%
}t \\
=\int_{\Omega _{T}}\int_{\mathcal{Y}_{2,2}}u_{0}\left( x,t\right) v\left(
x,t,y^{2},s^{2}\right) \emph{d}y^{2}\emph{d}s^{2}\emph{d}x\emph{d}t
\end{gather*}%
for all $v\in L^{2}(\Omega _{T};C_{\sharp }(\mathcal{Y}_{2,2}))$ and since $%
L^{2}(\Omega _{T})\subset L^{2}(\Omega _{T};C_{\sharp }(\mathcal{Y}_{2,2}))$
it follows that%
\begin{gather*}
\lim_{\varepsilon \rightarrow 0}\int_{\Omega _{T}}u_{\varepsilon }\left(
x,t\right) v\left( x,t\right) \emph{d}x\emph{d}t=\int_{\Omega _{T}}\int_{%
\mathcal{Y}_{2,2}}u_{0}\left( x,t\right) v\left( x,t\right) \emph{d}y^{2}%
\emph{d}s^{2}\emph{d}x\emph{d}t \\
=\int_{\Omega _{T}}u_{0}\left( x,t\right) v\left( x,t\right) \emph{d}x\emph{d%
}t\text{,}
\end{gather*}%
for all $v\in L^{2}(\Omega _{T})$. Observing that the weak convergence (\ref%
{svag konvergens}) implies%
\begin{equation*}
u_{\varepsilon }\left( x,t\right) \rightharpoonup u\left( x,t\right) \text{
in }L^{2}(\Omega _{T})
\end{equation*}%
for the same $u\in L^{2}(0,T;H_{0}^{1}(\Omega ))$ we see that $u_{0}$
coincides with the weak limit $u$, hence $u_{0}\in
L^{2}(0,T;H_{0}^{1}(\Omega ))$ and the proof of (\ref{3,3}) is complete.

Now we will identify $\tau _{0}$. Let $H$ denote the space of generalized
divergence-free functions in $L^{2}(\Omega ;L_{\sharp }^{2}(Y^{2})^{N})$
defined as%
\begin{equation*}
H=\left\{ v\in L^{2}(\Omega ;L_{\sharp }^{2}(Y^{2})^{N}):\nabla
_{y_{2}}\cdot v\left( x,y^{2}\right) =0\text{ and }\int_{Y_{2}}\nabla
_{y_{1}}\cdot v\left( x,y^{2}\right) \emph{d}y_{2}=0\right\} \text{.}
\end{equation*}%
Using $vc$, where $v\in \emph{D}(\Omega ;C_{\sharp }^{\infty
}(Y^{2}))^{N}\cap H$ and $c\in \emph{D}(0,T;C_{\sharp }^{\infty }(S^{2}))$,
as a test function in (\ref{tau}) we get, up to a subsequence,%
\begin{gather*}
\lim_{\varepsilon \rightarrow 0}\int_{\Omega _{T}}\nabla u_{\varepsilon
}\left( x,t\right) \cdot v\left( x,\frac{x}{\varepsilon },\frac{x}{%
\varepsilon ^{2}}\right) c\left( t,\frac{t}{\varepsilon ^{q}},\frac{t}{%
\varepsilon ^{r}}\right) \emph{d}x\emph{d}t \\
=\int_{\Omega _{T}}\int_{\mathcal{Y}_{2,2}}\tau _{0}\left(
x,t,y^{2},s^{2}\right) \cdot v\left( x,y^{2}\right) c\left( t,s^{2}\right) 
\emph{d}y^{2}\emph{d}s^{2}\emph{d}x\emph{d}t\text{,}
\end{gather*}%
for some $\tau _{0}\in L^{2}(\Omega _{T}\times \mathcal{Y}_{2,2})^{N}$. By
integration by parts in the left-hand side we obtain%
\begin{gather*}
\lim_{\varepsilon \rightarrow 0}\int_{\Omega _{T}}-u_{\varepsilon }\left(
x,t\right) \nabla \cdot v\left( x,\frac{x}{\varepsilon },\frac{x}{%
\varepsilon ^{2}}\right) c\left( t,\frac{t}{\varepsilon ^{q}},\frac{t}{%
\varepsilon ^{r}}\right) \emph{d}x\emph{d}t \\
=\lim_{\varepsilon \rightarrow 0}\int_{\Omega _{T}}-u_{\varepsilon }\left(
x,t\right) \left( \nabla _{x}\cdot v\left( x,\frac{x}{\varepsilon },\frac{x}{%
\varepsilon ^{2}}\right) +\varepsilon ^{-1}\nabla _{y_{1}}\cdot v\left( x,%
\frac{x}{\varepsilon },\frac{x}{\varepsilon ^{2}}\right) \right. \\
+\left. \varepsilon ^{-2}\nabla _{y_{2}}\cdot v\left( x,\frac{x}{\varepsilon 
},\frac{x}{\varepsilon ^{2}}\right) \right) c\left( t,\frac{t}{\varepsilon
^{q}},\frac{t}{\varepsilon ^{r}}\right) \emph{d}x\emph{d}t \\
=\lim_{\varepsilon \rightarrow 0}\int_{\Omega _{T}}-u_{\varepsilon }\left(
x,t\right) \left( \nabla _{x}\cdot v\left( x,\frac{x}{\varepsilon },\frac{x}{%
\varepsilon ^{2}}\right) +\varepsilon ^{-1}\nabla _{y_{1}}\cdot v\left( x,%
\frac{x}{\varepsilon },\frac{x}{\varepsilon ^{2}}\right) \right) \\
\times c\left( t,\frac{t}{\varepsilon ^{q}},\frac{t}{\varepsilon ^{r}}%
\right) \emph{d}x\emph{d}t\text{,}
\end{gather*}%
where the last term has vanished due to the fact that $\nabla _{y_{2}}\cdot
v=0$. Since%
\begin{equation*}
\int_{Y_{2}}\nabla _{y_{1}}\cdot v\left( x,y^{2}\right) \emph{d}y_{2}=0\text{%
,}
\end{equation*}%
Theorem 3.3 in \cite{AlBr} gives that $\left\{ \varepsilon ^{-2}\nabla
_{y_{1}}\cdot v\left( x,x/\varepsilon ,x/\varepsilon ^{2}\right) \right\} $
is bounded in $H^{-1}(\Omega )$. Passing to the limit while using this
boundedness yields%
\begin{gather*}
\int_{\Omega _{T}}\int_{\mathcal{Y}_{2,2}}-u\left( x,t\right) \nabla
_{x}\cdot v\left( x,y^{2}\right) c\left( t,s^{2}\right) \emph{d}y^{2}\emph{d}%
s^{2}\emph{d}x\emph{d}t \\
=\int_{\Omega _{T}}\int_{\mathcal{Y}_{2,2}}\nabla u\left( x,t\right) \cdot
v\left( x,y^{2}\right) c\left( t,s^{2}\right) \emph{d}y^{2}\emph{d}s^{2}%
\emph{d}x\emph{d}t\text{,}
\end{gather*}%
for all $v\in \emph{D}(\Omega ;C_{\sharp }^{\infty }(Y^{2}))^{N}\cap H$ and $%
c\in \emph{D}(0,T;C_{\sharp }^{\infty }(S^{2}))$. We conclude that%
\begin{gather*}
\int_{\Omega _{T}}\int_{\mathcal{Y}_{2,2}}\tau _{0}\left(
x,t,y^{2},s^{2}\right) \cdot v\left( x,y^{2}\right) c\left( t,s^{2}\right) 
\emph{d}y^{2}\emph{d}s^{2}\emph{d}x\emph{d}t \\
=\int_{\Omega _{T}}\int_{\mathcal{Y}_{2,2}}\nabla u\left( x,t\right) \cdot
v\left( x,y^{2}\right) c\left( t,s^{2}\right) \emph{d}y^{2}\emph{d}s^{2}%
\emph{d}x\emph{d}t
\end{gather*}%
or equivalently%
\begin{equation*}
\int_{\Omega _{T}}\int_{\mathcal{Y}_{2,2}}\left( \tau _{0}\left(
x,t,y^{2},s^{2}\right) -\nabla u\left( x,t\right) \right) \cdot v\left(
x,y^{2}\right) c\left( t,s^{2}\right) \emph{d}y^{2}\emph{d}s^{2}\emph{d}x%
\emph{d}t=0\text{.}
\end{equation*}%
By the Variational Lemma we obtain%
\begin{equation*}
\int_{\Omega }\int_{Y^{2}}\left( \tau _{0}\left( x,t,y^{2},s^{2}\right)
-\nabla u\left( x,t\right) \right) \cdot v\left( x,y^{2}\right) \emph{d}y^{2}%
\emph{d}x=0\text{,}
\end{equation*}%
a.e. in $\left( 0,T\right) \times S^{2}$. This means that $\tau _{0}-\nabla
u $ belongs to the orthogonal of $\emph{D}(\Omega ;C_{\sharp }^{\infty
}(Y^{2}))^{N}\cap H$ and by density (see property $\left( i\right) $ of
Lemma 3.7 in \cite{AlBr}) to the orthogonal of the whole space $H$.
According to property $(ii)$ of Lemma 3.7 in \cite{AlBr}, we deduce that%
\begin{equation*}
\tau _{0}\left( x,t,y^{2},s^{2}\right) -\nabla u\left( x,t\right) =\nabla
_{y_{1}}u_{1}\left( x,t,y_{1},s^{2}\right) +\nabla _{y_{2}}u_{2}\left(
x,t,y^{2},s^{2}\right)
\end{equation*}%
where $u_{1}\in L^{2}(\Omega _{T}\times S^{2};H_{\sharp }^{1}(Y_{1})/%
\mathbb{R}
)$ and $u_{2}\in L^{2}(\Omega _{T}\times \mathcal{Y}_{1,2};H_{\sharp
}^{1}(Y_{2})/%
\mathbb{R}
)$, which proves (\ref{gradsplit2}).

Now, choosing a test function $v\in L^{2}(\Omega _{T};C_{\sharp }(\mathcal{Y}%
_{1,2}))$ in the left-hand side of (\ref{gradsplit1}), (\ref{gradsplit2})
gives%
\begin{gather*}
\lim_{\varepsilon \rightarrow 0}\int_{\Omega _{T}}\nabla u_{\varepsilon
}\left( x,t\right) v\left( x,t,\frac{x}{\varepsilon },\frac{t}{\varepsilon
^{q}},\frac{t}{\varepsilon ^{r}}\right) \emph{d}x\emph{d}t \\
=\int_{\Omega _{T}}\int_{\mathcal{Y}_{2,2}}\left( \nabla u\left( x,t\right)
+\nabla _{y_{1}}u_{1}\left( x,t,y_{1},s^{2}\right) +\nabla
_{y_{2}}u_{2}\left( x,t,y^{2},s^{2}\right) \right) \\
\times v\left( x,t,y_{1},s^{2}\right) \emph{d}y^{2}\emph{d}s^{2}\emph{d}x%
\emph{d}t\text{.}
\end{gather*}%
Integrating over $Y_{2}$ while using the fact that%
\begin{equation*}
\int_{Y_{2}}\nabla _{y_{2}}u_{2}\left( x,t,y^{2},s^{2}\right) \emph{d}y_{2}=0
\end{equation*}%
we arrive at%
\begin{equation*}
\int_{\Omega _{T}}\int_{\mathcal{Y}_{1,2}}\left( \nabla u\left( x,t\right)
+\nabla _{y_{1}}u_{1}\left( x,t,y_{1},s^{2}\right) \right) v\left(
x,t,y_{1},s^{2}\right) \emph{d}y_{1}\emph{d}s^{2}\emph{d}x\emph{d}t\text{,}
\end{equation*}%
which proves (\ref{gradsplit1}).
\end{proof}

In the case of appearance of sequences that are not bounded in any Lebesgue
space, it might not be possible to obtain a multiscale limit. In \cite{Hol},
Holmbom introduced a concept of convergence that was improved by Nguetseng
and Woukeng in \cite{NgWo} and further developed and named very weak
multiscale convergence in \cite{FHOPveryweak}. The full generalization of
the concept was given in \cite{FHOLParbitrary}, for which we provide the
definition. This kind of convergence is crucial in the homogenization of (%
\ref{problem}), where unbounded sequences appear.

\begin{definition}
A sequence $\left\{ w_{\varepsilon }\right\} $ in $L^{1}(\Omega _{T})$ is
said to $(n+1,m+1)$-scale converge very weakly to $w_{0}\in L^{1}(\Omega
_{T}\times \mathcal{Y}_{n,m})$ if%
\begin{gather*}
\lim_{\varepsilon \rightarrow 0}\int_{\Omega _{T}}w_{\varepsilon }\left(
x,t\right) v_{1}\left( x,\frac{x}{\varepsilon _{1}},\ldots ,\frac{x}{%
\varepsilon _{n-1}}\right) v_{2}\left( \frac{x}{\varepsilon _{n}}\right)
c\left( t,\frac{t}{\varepsilon _{1}^{\prime }},\ldots ,\frac{t}{\varepsilon
_{m}^{\prime }}\right) dxdt \\
=\int_{\Omega _{T}}\int_{\mathcal{Y}_{n,m}}w_{0}\left(
x,t,y^{n},s^{m}\right) v_{1}\left( x,y^{n-1}\right) v_{2}\left( y_{n}\right)
c(t,s^{m})dy^{n}ds^{m}dxdt
\end{gather*}%
for any $v_{1}\in D(\Omega ;C_{\sharp }^{\infty }(Y^{n-1})),$ $v_{2}\in
C_{\sharp }^{\infty }(Y_{n})/%
\mathbb{R}
$ and $c\in D(0,T;C_{\sharp }^{\infty }(S^{m}))$, where 
\begin{equation}
\int_{Y_{n}}w_{0}\left( x,t,y^{n},s^{m}\right) dy_{n}=0\text{.}
\label{very weak unik}
\end{equation}%
We write%
\begin{equation*}
w_{\varepsilon }\left( x,t\right) \underset{vw}{\overset{n+1,m+1}{%
\rightharpoonup }}w_{0}\left( x,t,y^{n},s^{m}\right) \text{.}
\end{equation*}
\end{definition}

\begin{remark}
Due to (\ref{very weak unik}) the limit is unique.
\end{remark}

In earlier works, see e.g. \cite{PerPhD} or \cite{FHOLParbitrary},
compactness results for very weak evolution multiscale convergence for $%
\left\{ u_{\varepsilon }\right\} $ bounded in $W^{1,2}(0,T;H_{0}^{1}(\Omega
),L^{2}(\Omega ))$ have been established. Here, we will prove analogous
results without requiring boundedness of the time derivative in $%
L^{2}(0,T;H^{-1}(\Omega ))$. Note that the conditions (\ref{villkor1-2}) and
(\ref{villkor2-2}) are the same as (\ref{villkor1}) and (\ref{villkor2}) in
Theorem \ref{Theorem - gradient charact}.

\begin{theorem}
\label{Theorem - Compactresult vw} Let $\left\{ u_{\varepsilon }\right\} $
be a bounded sequence in $L^{2}(0,T;H_{0}^{1}(\Omega ))$ and, for any $v\in
D(\Omega )$, $c_{1}\in D(0,T)$, $c_{2}\in C_{\sharp }^{\infty }(S_{1})$, $%
c_{3}\in C_{\sharp }^{\infty }(S_{2})$ and $r>q>0$,%
\begin{equation}
\lim_{\varepsilon \rightarrow 0}\int_{\Omega _{T}}u_{\varepsilon }\left(
x,t\right) v\left( x\right) \partial _{t}\left( \varepsilon ^{r}c_{1}\left(
t\right) c_{2}\left( \frac{t}{\varepsilon ^{q}}\right) c_{3}\left( \frac{t}{%
\varepsilon ^{r}}\right) \right) dxdt=0  \label{villkor1-2}
\end{equation}%
and%
\begin{equation}
\lim_{\varepsilon \rightarrow 0}\int_{\Omega _{T}}u_{\varepsilon }\left(
x,t\right) v\left( x\right) \partial _{t}\left( \varepsilon ^{q}c_{1}\left(
t\right) c_{2}\left( \frac{t}{\varepsilon ^{q}}\right) \right) dxdt=0\text{.}
\label{villkor2-2}
\end{equation}%
Then, with $\varepsilon _{1}=\varepsilon $, $\varepsilon _{2}=\varepsilon
^{2}$, $\varepsilon _{1}^{\prime }=\varepsilon ^{q}$ and $\varepsilon
_{2}^{\prime }=\varepsilon ^{r}$, up to a subsequence%
\begin{equation}
\varepsilon ^{-1}u_{\varepsilon }\left( x,t\right) \underset{vw}{\overset{2,3%
}{\rightharpoonup }}u_{1}\left( x,t,y_{1},s^{2}\right)  \label{vw1}
\end{equation}%
and%
\begin{equation}
\varepsilon ^{-2}u_{\varepsilon }\left( x,t\right) \underset{vw}{\overset{3,3%
}{\rightharpoonup }}u_{2}\left( x,t,y^{2},s^{2}\right) \text{,}  \label{vw2}
\end{equation}%
where $u_{1}\in L^{2}(\Omega _{T}\times S^{2};H_{\sharp }^{1}(Y_{1})/%
\mathbb{R}
)$ and $u_{2}\in L^{2}(\Omega _{T}\times \mathcal{Y}_{1,2};H_{\sharp
}^{1}(Y_{2})/%
\mathbb{R}
)$ are the same as in (\ref{gradsplit1}) and (\ref{gradsplit2}) in Theorem %
\ref{Theorem - gradient charact}.
\end{theorem}

\begin{proof}
We point out that the task to prove (\ref{vw1}) and (\ref{vw2}) is to show%
\begin{gather}
\lim_{\varepsilon \rightarrow 0}\int_{\Omega _{T}}\varepsilon
^{-1}u_{\varepsilon }\left( x,t\right) v_{1}\left( x\right) v_{2}\left( 
\frac{x}{\varepsilon }\right) c\left( t,\frac{t}{\varepsilon ^{q}},\frac{t}{%
\varepsilon ^{r}}\right) \emph{d}x\emph{d}t  \label{vw1-2} \\
=\int_{\Omega _{T}}\int_{\mathcal{Y}_{1,2}}u_{1}\left(
x,t,y_{1},s^{2}\right) v_{1}\left( x\right) v_{2}\left( y_{1}\right) c\left(
t,s^{2}\right) \emph{d}y_{1}\emph{d}s^{2}\emph{d}x\emph{d}t\text{,}  \notag
\end{gather}%
for any $v_{1}\in \emph{D}(\Omega )$, $v_{2}\in C_{\sharp }^{\infty }(Y_{1})/%
\mathbb{R}
$ and $c\in \emph{D}(0,T;C_{\sharp }^{\infty }(S^{2}))$, and%
\begin{gather}
\lim_{\varepsilon \rightarrow 0}\int_{\Omega _{T}}\varepsilon
^{-2}u_{\varepsilon }\left( x,t\right) v_{1}\left( x,\frac{x}{\varepsilon }%
\right) v_{2}\left( \frac{x}{\varepsilon ^{2}}\right) c\left( t,\frac{t}{%
\varepsilon ^{q}},\frac{t}{\varepsilon ^{r}}\right) \emph{d}x\emph{d}t
\label{vw2-2} \\
=\int_{\Omega _{T}}\int_{\mathcal{Y}_{2,2}}u_{2}\left(
x,t,y^{2},s^{2}\right) v_{1}\left( x,y_{1}\right) v_{2}\left( y_{2}\right)
c\left( t,s^{2}\right) \emph{d}y^{2}\emph{d}s^{2}\emph{d}x\emph{d}t\text{,} 
\notag
\end{gather}%
for any $v_{1}\in \emph{D}(\Omega ;C_{\sharp }^{\infty }(Y_{1}))$, $v_{2}\in
C_{\sharp }^{\infty }(Y_{2})/%
\mathbb{R}
$ and $c\in \emph{D}(0,T;C_{\sharp }^{\infty }(S^{2}))$, respectively.

We start by proving (\ref{vw1}). Note that any $v_{2}\in C_{\sharp }^{\infty
}(Y_{1})/%
\mathbb{R}
$ can be represented by%
\begin{equation*}
v_{2}\left( y_{1}\right) =\Delta _{y_{1}}\rho \left( y_{1}\right) =\nabla
_{y_{1}}\cdot \left( \nabla _{y_{1}}\rho \left( y_{1}\right) \right)
\end{equation*}%
for some $\rho \in C_{\sharp }^{\infty }(Y_{1})/%
\mathbb{R}
$. The left-hand side of (\ref{vw1-2}) can now be expressed as%
\begin{gather*}
\lim_{\varepsilon \rightarrow 0}\int_{\Omega _{T}}\varepsilon
^{-1}u_{\varepsilon }\left( x,t\right) v_{1}\left( x\right) \nabla
_{y_{1}}\cdot \left( \nabla _{y_{1}}\rho \left( \frac{x}{\varepsilon }%
\right) \right) c\left( t,\frac{t}{\varepsilon ^{q}},\frac{t}{\varepsilon
^{r}}\right) \emph{d}x\emph{d}t \\
=\lim_{\varepsilon \rightarrow 0}\int_{\Omega _{T}}u_{\varepsilon }\left(
x,t\right) v_{1}\left( x\right) \nabla \cdot \left( \nabla _{y_{1}}\rho
\left( \frac{x}{\varepsilon }\right) \right) c\left( t,\frac{t}{\varepsilon
^{q}},\frac{t}{\varepsilon ^{r}}\right) \emph{d}x\emph{d}t \\
=\lim_{\varepsilon \rightarrow 0}\left( \int_{\Omega _{T}}-\nabla
u_{\varepsilon }\left( x,t\right) v_{1}\left( x\right) \cdot \nabla
_{y_{1}}\rho \left( \frac{x}{\varepsilon }\right) c\left( t,\frac{t}{%
\varepsilon ^{q}},\frac{t}{\varepsilon ^{r}}\right) \emph{d}x\emph{d}t\right.
\\
-\left. \int_{\Omega _{T}}u_{\varepsilon }\left( x,t\right) \nabla
v_{1}\left( x\right) \cdot \nabla _{y_{1}}\rho \left( \frac{x}{\varepsilon }%
\right) c\left( t,\frac{t}{\varepsilon ^{q}},\frac{t}{\varepsilon ^{r}}%
\right) \emph{d}x\emph{d}t\right) \text{,}
\end{gather*}%
where we used antidifferentiation with respect to $y_{1}$ and integration by
parts. By Theorem \ref{Theorem - gradient charact}, as $\varepsilon $ tends
to zero we obtain%
\begin{gather*}
\int_{\Omega _{T}}\int_{\mathcal{Y}_{1,2}}-\left( \nabla u\left( x,t\right)
+\nabla _{y_{1}}u_{1}\left( x,t,y_{1},s^{2}\right) \right) v_{1}\left(
x\right) \cdot \nabla _{y_{1}}\rho \left( y_{1}\right) c\left(
t,s^{2}\right) \emph{d}y_{1}\emph{d}s^{2}\emph{d}x\emph{d}t \\
-\int_{\Omega _{T}}\int_{\mathcal{Y}_{1,2}}u\left( x,t\right) \nabla
v_{1}\left( x\right) \cdot \nabla _{y_{1}}\rho \left( y_{1}\right) c\left(
t,s^{2}\right) \emph{d}y_{1}\emph{d}s^{2}\emph{d}x\emph{d}t\text{.}
\end{gather*}%
Integration by parts in the last term with respect to $x$ leaves us with%
\begin{equation*}
\int_{\Omega _{T}}\int_{\mathcal{Y}_{1,2}}-\nabla _{y_{1}}u_{1}\left(
x,t,y_{1},s^{2}\right) v_{1}\left( x\right) \cdot \nabla _{y_{1}}\rho \left(
y_{1}\right) c\left( t,s^{2}\right) \emph{d}y_{1}\emph{d}s^{2}\emph{d}x\emph{%
d}t
\end{equation*}%
and by integration by parts with respect to $y_{1}$ we arrive at%
\begin{gather*}
\int_{\Omega _{T}}\int_{\mathcal{Y}_{1,2}}u_{1}\left( x,t,y_{1},s^{2}\right)
v_{1}\left( x\right) \nabla _{y_{1}}\cdot \left( \nabla _{y_{1}}\rho \left(
y_{1}\right) \right) c\left( t,s^{2}\right) \emph{d}y_{1}\emph{d}s^{2}\emph{d%
}x\emph{d}t \\
=\int_{\Omega _{T}}\int_{\mathcal{Y}_{1,2}}u_{1}\left(
x,t,y_{1},s^{2}\right) v_{1}\left( x\right) v_{2}\left( y_{1}\right) c\left(
t,s^{2}\right) \emph{d}y_{1}\emph{d}s^{2}\emph{d}x\emph{d}t\text{,}
\end{gather*}%
which proves (\ref{vw1}).

We continue by proving (\ref{vw2}). Observing that any $v_{2}\in C_{\sharp
}^{\infty }(Y_{2})/%
\mathbb{R}
$ can be expressed as%
\begin{equation*}
v_{2}\left( y_{2}\right) =\Delta _{y_{2}}\rho \left( y_{2}\right) =\nabla
_{y_{2}}\cdot \left( \nabla _{y_{2}}\rho \left( y_{2}\right) \right)
\end{equation*}%
for some $\rho \in C_{\sharp }^{\infty }(Y_{2})/%
\mathbb{R}
$, following the same steps as above the left-hand side of (\ref{vw2-2}) can
be written%
\begin{gather*}
\lim_{\varepsilon \rightarrow 0}\left( \int_{\Omega _{T}}-\nabla
u_{\varepsilon }\left( x,t\right) v_{1}\left( x,\frac{x}{\varepsilon }%
\right) \cdot \nabla _{y_{2}}\rho \left( \frac{x}{\varepsilon ^{2}}\right)
c\left( t,\frac{t}{\varepsilon ^{q}},\frac{t}{\varepsilon ^{r}}\right) \emph{%
d}x\emph{d}t\right. \\
-\left. \int_{\Omega _{T}}u_{\varepsilon }\left( x,t\right) \left( \nabla
_{x}v_{1}\left( x,\frac{x}{\varepsilon }\right) +\varepsilon ^{-1}\nabla
_{y_{1}}v_{1}\left( x,\frac{x}{\varepsilon }\right) \right) \right. \\
\cdot \left. \nabla _{y_{2}}\rho \left( \frac{x}{\varepsilon ^{2}}\right)
c\left( t,\frac{t}{\varepsilon ^{q}},\frac{t}{\varepsilon ^{r}}\right) \emph{%
d}x\emph{d}t\right) \text{.}
\end{gather*}%
Since $\left\{ \varepsilon ^{-2}\nabla _{y_{1}}v_{1}\left( x,x/\varepsilon
\right) \cdot \nabla _{y_{2}}\rho \left( x/\varepsilon ^{2}\right) \right\} $
is bounded in $H^{-1}(\Omega )$, the last term in the second integral
vanishes as we pass to the limit and, applying Theorem \ref{Theorem -
gradient charact}, we obtain%
\begin{gather*}
\int_{\Omega _{T}}\int_{\mathcal{Y}_{2,2}}-\left( \nabla u\left( x,t\right)
+\nabla _{y_{1}}u_{1}\left( x,t,y_{1},s^{2}\right) +\nabla
_{y_{2}}u_{2}\left( x,t,y^{2},s^{2}\right) \right) \\
\times v_{1}\left( x,y_{1}\right) \cdot \nabla _{y_{2}}\rho \left(
y_{2}\right) c\left( t,s^{2}\right) \emph{d}y^{2}\emph{d}s^{2}\emph{d}x\emph{%
d}t \\
-\int_{\Omega _{T}}\int_{\mathcal{Y}_{2,2}}u\left( x,t\right) \nabla
_{x}v_{1}\left( x,y_{1}\right) \cdot \nabla _{y_{2}}\rho \left( y_{2}\right)
c\left( t,s^{2}\right) \emph{d}y^{2}\emph{d}s^{2}\emph{d}x\emph{d}t\text{.}
\end{gather*}%
By observing that%
\begin{equation*}
\int_{Y_{2}}\nabla _{y_{2}}\rho \left( y_{2}\right) dy_{2}=0\text{,}
\end{equation*}%
all but the last term in the first integral vanish, leaving us with%
\begin{equation*}
\int_{\Omega _{T}}\int_{\mathcal{Y}_{2,2}}-\nabla _{y_{2}}u_{2}\left(
x,t,y^{2},s^{2}\right) v_{1}\left( x,y_{1}\right) \cdot \nabla _{y_{2}}\rho
\left( y_{2}\right) c\left( t,s^{2}\right) \emph{d}y^{2}\emph{d}s^{2}\emph{d}%
x\emph{d}t
\end{equation*}%
and integration by parts with respect to $y_{2}$ gives%
\begin{gather*}
\int_{\Omega _{T}}\int_{\mathcal{Y}_{2,2}}u_{2}\left( x,t,y^{2},s^{2}\right)
v_{1}\left( x,y_{1}\right) \nabla _{y_{2}}\cdot \left( \nabla _{y_{2}}\rho
\left( y_{2}\right) \right) c\left( t,s^{2}\right) \emph{d}y^{2}\emph{d}s^{2}%
\emph{d}x\emph{d}t \\
=\int_{\Omega _{T}}\int_{\mathcal{Y}_{2,2}}u_{2}\left(
x,t,y^{2},s^{2}\right) v_{1}\left( x,y_{1}\right) v_{2}\left( y_{2}\right)
c\left( t,s^{2}\right) \emph{d}y^{2}\emph{d}s^{2}\emph{d}x\emph{d}t\text{,}
\end{gather*}%
which proves (\ref{vw2}).
\end{proof}

\section{Homogenization\label{Homogenization}}

This section is devoted to the homogenization of problem (\ref{problem}). We
start by recalling the equation%
\begin{eqnarray}
\varepsilon ^{p}\partial _{t}u_{\varepsilon }\left( x,t\right) -\nabla \cdot
\left( a\left( \frac{x}{\varepsilon },\frac{x}{\varepsilon ^{2}},\frac{t}{%
\varepsilon ^{q}},\frac{t}{\varepsilon ^{r}}\right) \nabla u_{\varepsilon
}\left( x,t\right) \right) &=&f\left( x,t\right) \text{ in }\Omega _{T}\text{%
,}  \notag \\
u_{\varepsilon }\left( x,0\right) &=&u_{0}\left( x\right) \text{ in }\Omega 
\text{,}  \label{problem2} \\
u_{\varepsilon }\left( x,t\right) &=&0\text{ on }\partial \Omega \times
\left( 0,T\right) \text{,}  \notag
\end{eqnarray}%
where $0<p<q<r$, $f\in L^{2}(\Omega _{T})$ and $u_{0}\in L^{2}(\Omega )$.
Under the assumption that the coefficient $a\in C_{\sharp }(\mathcal{Y}%
_{2,2})^{N\times N}$ satisfies the coercivity condition%
\begin{equation*}
a\left( y^{2},s^{2}\right) \xi \cdot \xi \geq C_{0}\left\vert \xi
\right\vert ^{2}
\end{equation*}%
for all $\left( y^{2},s^{2}\right) \in 
\mathbb{R}
^{2N}\times 
\mathbb{R}
^{2}$, all $\xi \in 
\mathbb{R}
^{N}$ and some $C_{0}>0$, (\ref{problem2}) possesses a unique solution $%
u_{\varepsilon }\in W^{1,2}(0,T;H_{0}^{1}(\Omega ),L^{2}(\Omega ))$ for
every fixed $\varepsilon $, see Section 23.7 in \cite{ZeiIIA}. Further, the
a priori estimate%
\begin{equation}
\left\Vert u_{\varepsilon }\right\Vert _{L^{2}(0,T;H_{0}^{1}(\Omega ))}\leq
C_{1}  \label{a priori}
\end{equation}%
holds for some $C_{1}>0$ independent of $\varepsilon $, according to the
reasoning in Section 3 in \cite{JoLoarxiv}.

Before we are ready to give the homogenization result we show that the
assumptions (\ref{villkor1}) and (\ref{villkor2}) in Theorems \ref{Theorem -
gradient charact} and \ref{Theorem - Compactresult vw} are satisfied, i.e.
that for $v\in \emph{D}(\Omega )$, $c_{1}\in \emph{D}(0,T)$,$\ c_{2}\in
C_{\sharp }^{\infty }(S_{1})$, $c_{3}\in C_{\sharp }^{\infty }(S_{2})$ and $%
r>q>0$%
\begin{equation}
\lim_{\varepsilon \rightarrow 0}\int_{\Omega _{T}}u_{\varepsilon }\left(
x,t\right) v\left( x\right) \partial _{t}\left( \varepsilon ^{r}c_{1}\left(
t\right) c_{2}\left( \frac{t}{\varepsilon ^{q}}\right) c_{3}\left( \frac{t}{%
\varepsilon ^{r}}\right) \right) \emph{d}x\emph{d}t=0  \label{villkor1-3}
\end{equation}%
and%
\begin{equation}
\lim_{\varepsilon \rightarrow 0}\int_{\Omega _{T}}u_{\varepsilon }\left(
x,t\right) v\left( x\right) \partial _{t}\left( \varepsilon ^{q}c_{1}\left(
t\right) c_{2}\left( \frac{t}{\varepsilon ^{q}}\right) \right) \emph{d}x%
\emph{d}t=0\text{.}  \label{villkor2-3}
\end{equation}%
The weak form of (\ref{problem2}) is%
\begin{gather}
\int_{\Omega _{T}}-\varepsilon ^{p}u_{\varepsilon }\left( x,t\right) v\left(
x\right) \partial _{t}c\left( t\right) +a\left( \frac{x}{\varepsilon },\frac{%
x}{\varepsilon ^{2}},\frac{t}{\varepsilon ^{q}},\frac{t}{\varepsilon ^{r}}%
\right) \nabla u_{\varepsilon }\left( x,t\right) \cdot \nabla v\left(
x\right) c\left( t\right) \emph{d}x\emph{d}t  \notag \\
=\int_{\Omega _{T}}f\left( x,t\right) v\left( x\right) c\left( t\right) 
\emph{d}x\emph{d}t\text{,}  \label{weak form}
\end{gather}%
where $0<p<q<r$, for all $v\in H_{0}^{1}(\Omega )$ and $c\in \emph{D}(0,T)$.
Taking the test function%
\begin{equation*}
v\left( x\right) c\left( t\right) =\varepsilon ^{r-p}v_{1}\left( x\right)
c_{1}\left( t\right) c_{2}\left( \frac{t}{\varepsilon ^{q}}\right)
c_{3}\left( \frac{t}{\varepsilon ^{r}}\right) \text{,}
\end{equation*}%
with $v_{1}\in \emph{D}(\Omega )$, $c_{1}\in \emph{D}(0,T)$,$\ c_{2}\in
C_{\sharp }^{\infty }(S_{1})$ and $c_{3}\in C_{\sharp }^{\infty }(S_{2})$,
we get, after rearranging,%
\begin{gather*}
\int_{\Omega _{T}}u_{\varepsilon }\left( x,t\right) v_{1}\left( x\right)
\partial _{t}\left( \varepsilon ^{r}c_{1}\left( t\right) c_{2}\left( \frac{t%
}{\varepsilon ^{q}}\right) c_{3}\left( \frac{t}{\varepsilon ^{r}}\right)
\right) \emph{d}x\emph{d}t \\
=\int_{\Omega _{T}}\varepsilon ^{r-p}a\left( \frac{x}{\varepsilon },\frac{x}{%
\varepsilon ^{2}},\frac{t}{\varepsilon ^{q}},\frac{t}{\varepsilon ^{r}}%
\right) \nabla u_{\varepsilon }\left( x,t\right) \cdot \nabla v_{1}\left(
x\right) c_{1}\left( t\right) c_{2}\left( \frac{t}{\varepsilon ^{q}}\right)
c_{3}\left( \frac{t}{\varepsilon ^{r}}\right) \emph{d}x\emph{d}t \\
-\int_{\Omega _{T}}\varepsilon ^{r-p}f\left( x,t\right) v_{1}\left( x\right)
c_{1}\left( t\right) c_{2}\left( \frac{t}{\varepsilon ^{q}}\right)
c_{3}\left( \frac{t}{\varepsilon ^{r}}\right) \emph{d}x\emph{d}t\text{.}
\end{gather*}%
Passing to the limit while recalling that $\left\{ u_{\varepsilon }\right\} $
is bounded in $L^{2}(0,T;H_{0}^{1}(\Omega ))$, which implies boundedness of $%
\left\{ \nabla u_{\varepsilon }\right\} $ in $L^{2}(\Omega _{T})^{N}$, we
obtain%
\begin{gather*}
\lim_{\varepsilon \rightarrow 0}\int_{\Omega _{T}}u_{\varepsilon }\left(
x,t\right) v_{1}\left( x\right) \partial _{t}\left( \varepsilon
^{r}c_{1}\left( t\right) c_{2}\left( \frac{t}{\varepsilon ^{q}}\right)
c_{3}\left( \frac{t}{\varepsilon ^{r}}\right) \right) \emph{d}x\emph{d}t \\
=\lim_{\varepsilon \rightarrow 0}\left( \int_{\Omega _{T}}\varepsilon
^{r-p}a\left( \frac{x}{\varepsilon },\frac{x}{\varepsilon ^{2}},\frac{t}{%
\varepsilon ^{q}},\frac{t}{\varepsilon ^{r}}\right) \nabla u_{\varepsilon
}\left( x,t\right) \right. \\
\cdot \nabla v_{1}\left( x\right) c_{1}\left( t\right) c_{2}\left( \frac{t}{%
\varepsilon ^{q}}\right) c_{3}\left( \frac{t}{\varepsilon ^{r}}\right) \emph{%
d}x\emph{d}t \\
-\left. \int_{\Omega _{T}}\varepsilon ^{r-p}f\left( x,t\right) v_{1}\left(
x\right) c_{1}\left( t\right) c_{2}\left( \frac{t}{\varepsilon ^{q}}\right)
c_{3}\left( \frac{t}{\varepsilon ^{r}}\right) \emph{d}x\emph{d}t\right) =0
\end{gather*}%
and (\ref{villkor1-3}) is fulfilled. Following the same steps again but
taking the test function%
\begin{equation*}
v\left( x\right) c\left( t\right) =\varepsilon ^{q-p}v_{1}\left( x\right)
c_{1}\left( t\right) c_{2}\left( \frac{t}{\varepsilon ^{q}}\right) \text{,}
\end{equation*}%
where $v_{1}\in \emph{D}(\Omega )$, $c_{1}\in \emph{D}(0,T)$ and$\ c_{2}\in
C_{\sharp }^{\infty }(S_{1})$, in the weak form (\ref{weak form}) yields
that (\ref{villkor2-3}) is fulfilled.

We are now prepared to prove the homogenization result. Depending on the
choices of $p$, $q$ and $r$ $(0<p<q<r)$ in (\ref{problem2}), we get
different outcomes. In Theorem \ref{Theorem - homogenisering} we present the
13 possible cases, arising from different combinations of $p$, $q$ and $r$.
Here we will see that the local problems are parabolic when the matching
between the microscopic scales that give resonance is shifted by $p$
compared to the standard case (cf. Section 5.3.1 in \cite{PerPhD}). This
means that resonance appears when\ the temporal scale multiplied by $%
\varepsilon ^{-p}$ is the square of a spatial scale.

\begin{theorem}
\label{Theorem - homogenisering} Let $\left\{ u_{\varepsilon }\right\} $ be
a sequence of solutions to (\ref{problem2}) in $W^{1,2}(0,T;H_{0}^{1}(\Omega
),L^{2}(\Omega ))$. Then it holds that%
\begin{equation}
u_{\varepsilon }\left( x,t\right) \rightharpoonup u\left( x,t\right) \text{
in }L^{2}(0,T;H_{0}^{1}(\Omega ))  \label{svag homogenisering}
\end{equation}%
\begin{equation}
u_{\varepsilon }\left( x,t\right) \overset{3,3}{\rightharpoonup }u\left(
x,t\right)  \label{3,3 homogenisering}
\end{equation}%
and%
\begin{equation}
\nabla u_{\varepsilon }\left( x,t\right) \overset{3,3}{\rightharpoonup }%
\nabla u\left( x,t\right) +\nabla _{y_{1}}u_{1}\left( x,t,y_{1},s^{2}\right)
+\nabla _{y_{2}}u_{2}\left( x,t,y^{2},s^{2}\right) \text{,}
\label{gradsplit homogenisering}
\end{equation}%
where $u\in L^{2}(0,T;H_{0}^{1}(\Omega ))$ is the unique solution to the
homogenized problem%
\begin{eqnarray}
-\nabla \cdot \left( b\nabla u\left( x,t\right) \right) &=&f\left(
x,t\right) \text{ in }\Omega _{T}\text{,}  \label{homogeniserat problem} \\
u\left( x,t\right) &=&0\text{ on }\partial \Omega \times \left( 0,T\right) 
\text{,}  \notag
\end{eqnarray}%
where the coefficient $b$ is characterized by the formulas below. For all 13
cases we assume that $0<p<q<r$.

\begin{enumerate}
\item Letting $r<2+p$, the homogenized coefficient is given by%
\begin{gather}
b\nabla u\left( x,t\right) =\int_{\mathcal{Y}_{2,2}}a\left(
y^{2},s^{2}\right) \left( \nabla u\left( x,t\right) +\nabla
_{y_{1}}u_{1}\left( x,t,y_{1},s^{2}\right) \right.
\label{hom.coeff. inget oberoende} \\
+\left. \nabla _{y_{2}}u_{2}\left( x,t,y^{2},s^{2}\right) \right)
dy^{2}ds^{2}\text{,}  \notag
\end{gather}%
and $u_{1}\in L^{2}(\Omega _{T}\times S^{2};H_{\sharp }^{1}(Y_{1})/%
\mathbb{R}
)$ and $u_{2}\in L^{2}(\Omega _{T}\times \mathcal{Y}_{1,2};H_{\sharp
}^{1}(Y_{2})/%
\mathbb{R}
)$ are given by the local problems%
\begin{gather}
-\nabla _{y_{2}}\cdot \left( a\left( y^{2},s^{2}\right) \left( \nabla
u\left( x,t\right) +\nabla _{y_{1}}u_{1}\left( x,t,y_{1},s^{2}\right)
\right. \right.  \label{first local case 1} \\
+\left. \left. \nabla _{y_{2}}u_{2}\left( x,t,y^{2},s^{2}\right) \right)
\right) =0  \notag
\end{gather}%
and%
\begin{gather}
-\nabla _{y_{1}}\cdot \int_{Y_{2}}a\left( y^{2},s^{2}\right) \left( \nabla
u\left( x,t\right) +\nabla _{y_{1}}u_{1}\left( x,t,y_{1},s^{2}\right) \right.
\label{second local case 1} \\
+\left. \nabla _{y_{2}}u_{2}\left( x,t,y^{2},s^{2}\right) \right) dy_{2}=0%
\text{.}  \notag
\end{gather}

\item Choosing $r=2+p$, the coefficient $b$ is determined by (\ref%
{hom.coeff. inget oberoende}) while $u_{1}\in L^{2}(\Omega _{T}\times S_{1};%
\mathcal{W}_{1,2})$ and $u_{2}\in L^{2}(\Omega _{T}\times \mathcal{Y}%
_{1,2};H_{\sharp }^{1}(Y_{2})/%
\mathbb{R}
)$ are the solutions to the local problems%
\begin{gather}
-\nabla _{y_{2}}\cdot \left( a\left( y^{2},s^{2}\right) \left( \nabla
u\left( x,t\right) +\nabla _{y_{1}}u_{1}\left( x,t,y_{1},s^{2}\right)
\right. \right.  \label{first local case 2} \\
+\left. \left. \nabla _{y_{2}}u_{2}\left( x,t,y^{2},s^{2}\right) \right)
\right) =0  \notag
\end{gather}%
and%
\begin{gather}
\partial _{s_{2}}u_{1}\left( x,t,y_{1},s^{2}\right) -\nabla _{y_{1}}\cdot
\int_{Y_{2}}a\left( y^{2},s^{2}\right) \left( \nabla u\left( x,t\right)
+\nabla _{y_{1}}u_{1}\left( x,t,y_{1},s^{2}\right) \right.
\label{second local case 2} \\
+\left. \nabla _{y_{2}}u_{2}\left( x,t,y^{2},s^{2}\right) \right) dy_{2}=0%
\text{.}  \notag
\end{gather}

\item If $2+p<r<4+p$ while $q<2+p$, we have%
\begin{gather}
b\nabla u\left( x,t\right) =\int_{\mathcal{Y}_{2,2}}a\left(
y^{2},s^{2}\right) \left( \nabla u\left( x,t\right) +\nabla
_{y_{1}}u_{1}\left( x,t,y_{1},s_{1}\right) \right.
\label{hom.coeff. u1 ober s1} \\
+\left. \nabla _{y_{2}}u_{2}\left( x,t,y^{2},s^{2}\right) \right)
dy^{2}ds^{2}  \notag
\end{gather}%
where $u_{1}\in L^{2}(\Omega _{T}\times S_{1};H_{\sharp }^{1}(Y_{1})/%
\mathbb{R}
)$ and $u_{2}\in L^{2}(\Omega _{T}\times \mathcal{Y}_{1,2};H_{\sharp
}^{1}(Y_{2})/%
\mathbb{R}
)$ are given by the system%
\begin{gather}
-\nabla _{y_{2}}\cdot \left( a\left( y^{2},s^{2}\right) \left( \nabla
u\left( x,t\right) +\nabla _{y_{1}}u_{1}\left( x,t,y_{1},s_{1}\right)
\right. \right.  \label{first local case 3} \\
+\left. \left. \nabla _{y_{2}}u_{2}\left( x,t,y^{2},s^{2}\right) \right)
\right) =0  \notag
\end{gather}%
and%
\begin{gather}
\nabla _{y_{1}}\cdot \int_{Y_{2}\times S_{2}}a\left( y^{2},s^{2}\right)
\left( \nabla u\left( x,t\right) +\nabla _{y_{1}}u_{1}\left(
x,t,y_{1},s_{1}\right) \right.  \label{second local case 3} \\
+\left. \nabla _{y_{2}}u_{2}\left( x,t,y^{2},s^{2}\right) \right)
dy_{2}ds_{2}=0\text{.}  \notag
\end{gather}

\item Taking $r<4+p$ and $q=2+p$, the homogenized coefficient is given by (%
\ref{hom.coeff. u1 ober s1}) and $u_{1}\in L^{2}(\Omega _{T};\mathcal{W}%
_{1,1})$ and $u_{2}\in L^{2}(\Omega _{T}\times \mathcal{Y}_{1,2};H_{\sharp
}^{1}(Y_{2})/%
\mathbb{R}
)$ are determined by%
\begin{gather}
-\nabla _{y_{2}}\cdot \left( a\left( y^{2},s^{2}\right) \left( \nabla
u\left( x,t\right) +\nabla _{y_{1}}u_{1}\left( x,t,y_{1},s_{1}\right)
\right. \right.  \label{first local case 4} \\
+\left. \left. \nabla _{y_{2}}u_{2}\left( x,t,y^{2},s^{2}\right) \right)
\right) =0  \notag
\end{gather}%
and%
\begin{gather}
\partial _{s_{1}}u_{1}\left( x,t,y_{1},s_{1}\right) -\nabla _{y_{1}}\cdot
\int_{Y_{2}\times S_{2}}a\left( y^{2},s^{2}\right) \left( \nabla u\left(
x,t\right) +\nabla _{y_{1}}u_{1}\left( x,t,y_{1},s_{1}\right) \right.
\label{second local case 4} \\
+\left. \nabla _{y_{2}}u_{2}\left( x,t,y^{2},s^{2}\right) \right)
dy_{2}ds_{2}=0\text{.}  \notag
\end{gather}

\item When $r<4+p$ and $q>2+p$ the coefficient $b$ is determined by%
\begin{gather}
b\nabla u\left( x,t\right) =\int_{\mathcal{Y}_{2,2}}a\left(
y^{2},s^{2}\right) \left( \nabla u\left( x,t\right) +\nabla
_{y_{1}}u_{1}\left( x,t,y_{1}\right) \right.
\label{hom.coeff. u1 ober s1 och s2} \\
+\left. \nabla _{y_{2}}u_{2}\left( x,t,y^{2},s^{2}\right) \right)
dy^{2}ds^{2}  \notag
\end{gather}%
and the local problems are%
\begin{gather}
-\nabla _{y_{2}}\cdot \left( a\left( y^{2},s^{2}\right) \left( \nabla
u\left( x,t\right) +\nabla _{y_{1}}u_{1}\left( x,t,y_{1}\right) \right.
\right.  \label{first local case 5} \\
+\left. \nabla _{y_{2}}u_{2}\left( x,t,y^{2},s^{2}\right) \right) =0  \notag
\end{gather}%
and%
\begin{gather}
-\nabla _{y_{1}}\cdot \int_{Y_{2}\times S^{2}}a\left( y^{2},s^{2}\right)
\left( \nabla u\left( x,t\right) +\nabla _{y_{1}}u_{1}\left(
x,t,y_{1}\right) \right.  \label{second local case 5} \\
+\left. \nabla _{y_{2}}u_{2}\left( x,t,y^{2},s^{2}\right) \right)
dy_{2}ds^{2}=0\text{,}  \notag
\end{gather}%
where $u_{1}\in L^{2}(\Omega _{T};H_{\sharp }^{1}(Y_{1})/%
\mathbb{R}
)$ and $u_{2}\in L^{2}(\Omega _{T}\times \mathcal{Y}_{1,2};H_{\sharp
}^{1}(Y_{2})/%
\mathbb{R}
)$.

\item In the case when $r=4+p$ while $q<2+p$, the homogenized coefficient is
characterized by (\ref{hom.coeff. u1 ober s1}) while $u_{1}\in L^{2}(\Omega
_{T}\times S_{1};H_{\sharp }^{1}(Y_{1})/%
\mathbb{R}
)$ and $u_{2}\in L^{2}(\Omega _{T}\times \mathcal{Y}_{1,1};\mathcal{W}%
_{2,2}) $ are given by the system of local problems%
\begin{gather}
\partial _{s_{2}}u_{2}\left( x,t,y^{2},s^{2}\right) -\nabla _{y_{2}}\cdot
\left( a\left( y^{2},s^{2}\right) \left( \nabla u\left( x,t\right) +\nabla
_{y_{1}}u_{1}\left( x,t,y_{1},s_{1}\right) \right. \right.
\label{first local case 6} \\
+\left. \left. \nabla _{y_{2}}u_{2}\left( x,t,y^{2},s^{2}\right) \right)
\right) =0  \notag
\end{gather}%
and%
\begin{gather}
-\nabla _{y_{1}}\cdot \int_{Y_{2}\times S_{2}}a\left( y^{2},s^{2}\right)
\left( \nabla u\left( x,t\right) +\nabla _{y_{1}}u_{1}\left(
x,t,y_{1},s_{1}\right) \right.  \label{second local case 6} \\
+\left. \nabla _{y_{2}}u_{2}\left( x,t,y^{2},s^{2}\right) \right)
dy_{2}ds_{2}=0\text{.}  \notag
\end{gather}

\item When $r=4+p$ and $q=2+p$, the coefficient $b$ is given by (\ref%
{hom.coeff. u1 ober s1}) where $u_{1}\in L^{2}(\Omega _{T};\mathcal{W}%
_{1,1}) $ and $u_{2}\in L^{2}(\Omega _{T}\times \mathcal{Y}_{1,1};\mathcal{W}%
_{2,2})$ are the solutions to%
\begin{gather}
\partial _{s_{2}}u_{2}\left( x,t,y^{2},s^{2}\right) -\nabla _{y_{2}}\cdot
\left( a\left( y^{2},s^{2}\right) \left( \nabla u\left( x,t\right) +\nabla
_{y_{1}}u_{1}\left( x,t,y_{1},s_{1}\right) \right. \right.
\label{first local case 7} \\
+\left. \left. \nabla _{y_{2}}u_{2}\left( x,t,y^{2},s^{2}\right) \right)
\right) =0  \notag
\end{gather}%
and%
\begin{gather}
\partial _{s_{1}}u_{1}\left( x,t,y_{1},s_{1}\right) -\nabla _{y_{1}}\cdot
\int_{Y_{2}\times S_{2}}a\left( y^{2},s^{2}\right) \left( \nabla u\left(
x,t\right) +\nabla _{y_{1}}u_{1}\left( x,t,y_{1},s_{1}\right) \right.
\label{second local case 7} \\
+\left. \nabla _{y_{2}}u_{2}\left( x,t,y^{2},s^{2}\right) \right)
dy_{2}ds_{2}=0\text{.}  \notag
\end{gather}

\item Letting $r=4+p$ while $q>2+p$ gives us the homogenized coefficient (%
\ref{hom.coeff. u1 ober s1 och s2}) defined by the system of local problems%
\begin{gather}
\partial _{s_{2}}u_{2}\left( x,t,y^{2},s^{2}\right) -\nabla _{y_{2}}\cdot
\left( a\left( y^{2},s^{2}\right) \left( \nabla u\left( x,t\right) +\nabla
_{y_{1}}u_{1}\left( x,t,y_{1}\right) \right. \right.
\label{first local case 8} \\
+\left. \nabla _{y_{2}}u_{2}\left( x,t,y^{2},s^{2}\right) \right) =0  \notag
\end{gather}%
and%
\begin{gather}
-\nabla _{y_{1}}\cdot \int_{Y_{2}\times S^{2}}a\left( y^{2},s^{2}\right)
\left( \nabla u\left( x,t\right) +\nabla _{y_{1}}u_{1}\left(
x,t,y_{1}\right) \right.  \label{second local case 8} \\
+\left. \nabla _{y_{2}}u_{2}\left( x,t,y^{2},s^{2}\right) \right)
dy_{2}ds^{2}=0\text{,}  \notag
\end{gather}%
where $u_{1}\in L^{2}(\Omega _{T};H_{\sharp }^{1}(Y_{1})/%
\mathbb{R}
)$ and $u_{2}\in L^{2}(\Omega _{T}\times \mathcal{Y}_{1,1};\mathcal{W}%
_{2,2}) $.

\item Choosing $r>4+p$ and $q<2+p$, we have the homogenized coefficient%
\begin{gather}
b\nabla u\left( x,t\right) =\int_{\mathcal{Y}_{2,1}}\left(
\int_{S_{2}}a\left( y^{2},s^{2}\right) ds_{2}\right) \left( \nabla u\left(
x,t\right) +\nabla _{y_{1}}u_{1}\left( x,t,y_{1},s_{1}\right) \right.
\label{hom.coeff. u1 och u2 ober s2} \\
+\left. \nabla _{y_{2}}u_{2}\left( x,t,y^{2},s_{1}\right) \right)
dy^{2}ds_{1}  \notag
\end{gather}%
where $u_{1}\in L^{2}(\Omega _{T}\times S_{1};H_{\sharp }^{1}(Y_{1})/%
\mathbb{R}
)$ and $u_{2}\in L^{2}(\Omega _{T}\times \mathcal{Y}_{1,1};H_{\sharp
}^{1}(Y_{2})/%
\mathbb{R}
)$ are the solutions to the local problems%
\begin{gather}
-\nabla _{y_{2}}\cdot \left( \int_{S_{2}}a\left( y^{2},s^{2}\right)
ds_{2}\right) \left( \nabla u\left( x,t\right) +\nabla _{y_{1}}u_{1}\left(
x,t,y_{1},s_{1}\right) \right.  \label{first local case 9} \\
+\left. \nabla _{y_{2}}u_{2}\left( x,t,y^{2},s_{1}\right) \right) =0  \notag
\end{gather}%
and%
\begin{gather}
-\nabla _{y_{1}}\cdot \int_{Y_{2}}\left( \int_{S_{2}}a\left(
y^{2},s^{2}\right) ds_{2}\right) \left( \nabla u\left( x,t\right) +\nabla
_{y_{1}}u_{1}\left( x,t,y_{1},s_{1}\right) \right.
\label{second local case 9} \\
+\left. \nabla _{y_{2}}u_{2}\left( x,t,y^{2},s_{1}\right) \right) dy_{2}=0%
\text{.}  \notag
\end{gather}

\item When $r>4+p$ while $q=2+p$, the homogenized coefficient is given by (%
\ref{hom.coeff. u1 och u2 ober s2}) and the local problems are%
\begin{gather}
-\nabla _{y_{2}}\cdot \left( \int_{S_{2}}a\left( y^{2},s^{2}\right)
ds_{2}\right) \left( \nabla u\left( x,t\right) +\nabla _{y_{1}}u_{1}\left(
x,t,y_{1},s_{1}\right) \right.  \label{first local case 10} \\
+\left. \nabla _{y_{2}}u_{2}\left( x,t,y^{2},s_{1}\right) \right) =0  \notag
\end{gather}%
and%
\begin{gather}
\partial _{s_{1}}u_{1}\left( x,t,y_{1},s_{1}\right) -\nabla _{y_{1}}\cdot
\int_{Y_{2}}\left( \int_{S_{2}}a\left( y^{2},s^{2}\right) ds_{2}\right)
\left( \nabla u\left( x,t\right) \right.  \label{second local case 10} \\
+\left. \nabla _{y_{1}}u_{1}\left( x,t,y_{1},s_{1}\right) +\nabla
_{y_{2}}u_{2}\left( x,t,y^{2},s_{1}\right) \right) dy_{2}=0\text{,}  \notag
\end{gather}%
with $u_{1}\in L^{2}(\Omega _{T};\mathcal{W}_{1,1})$ and $u_{2}\in
L^{2}(\Omega _{T}\times \mathcal{Y}_{1,1};H_{\sharp }^{1}(Y_{2})/%
\mathbb{R}
)$.

\item When $r>4+p$ and $2+p<q<4+p$, we have%
\begin{gather}
b\nabla u\left( x,t\right) =\int_{\mathcal{Y}_{2,1}}\left(
\int_{S_{2}}a\left( y^{2},s^{2}\right) ds_{2}\right) \left( \nabla u\left(
x,t\right) +\nabla _{y_{1}}u_{1}\left( x,t,y_{1}\right) \right.
\label{hom.coeff. u1 ober s1 och s2 u2 ober s2} \\
+\left. \nabla _{y_{2}}u_{2}\left( x,t,y^{2},s_{1}\right) \right)
dy^{2}ds_{1}  \notag
\end{gather}%
together with the local problems%
\begin{gather}
-\nabla _{y_{2}}\cdot \left( \int_{S_{2}}a\left( y^{2},s^{2}\right)
ds_{2}\right) \left( \nabla u\left( x,t\right) +\nabla _{y_{1}}u_{1}\left(
x,t,y_{1}\right) \right.  \label{first local case 11} \\
+\left. \nabla _{y_{2}}u_{2}\left( x,t,y^{2},s_{1}\right) \right) =0  \notag
\end{gather}%
and%
\begin{gather}
-\nabla _{y_{1}}\cdot \int_{Y_{2}\times S_{1}}\left( \int_{S_{2}}a\left(
y^{2},s^{2}\right) ds_{2}\right) \left( \nabla u\left( x,t\right) +\nabla
_{y_{1}}u_{1}\left( x,t,y_{1}\right) \right.  \label{second local case 11} \\
+\left. \nabla _{y_{2}}u_{2}\left( x,t,y^{2},s_{1}\right) \right)
dy_{2}ds_{1}=0\text{,}  \notag
\end{gather}%
where $u_{1}\in L^{2}(\Omega _{T};H_{\sharp }^{1}(Y_{1})/%
\mathbb{R}
)$ and $u_{2}\in L^{2}(\Omega _{T}\times \mathcal{Y}_{1,1};H_{\sharp
}^{1}(Y_{2})/%
\mathbb{R}
)$.

\item Taking $q=4+p$, the coefficient in the homogenized problem is given by
(\ref{hom.coeff. u1 ober s1 och s2 u2 ober s2}) and $u_{1}\in L^{2}(\Omega
_{T};H_{\sharp }^{1}(Y_{1})/%
\mathbb{R}
)$ and $u_{2}\in L^{2}(\Omega _{T}\times Y_{1};\mathcal{W}_{2,1})$ are
determined by%
\begin{gather}
\partial _{s_{1}}u_{2}\left( x,t,y^{2},s_{1}\right) -\nabla _{y_{2}}\cdot
\left( \int_{S_{2}}a\left( y^{2},s^{2}\right) ds_{2}\right) \left( \nabla
u\left( x,t\right) +\nabla _{y_{1}}u_{1}\left( x,t,y_{1}\right) \right.
\label{first local case 12} \\
+\left. \nabla _{y_{2}}u_{2}\left( x,t,y^{2},s_{1}\right) \right) =0  \notag
\end{gather}%
and%
\begin{gather}
-\nabla _{y_{1}}\cdot \int_{Y_{2}\times S_{1}}\left( \int_{S_{2}}a\left(
y^{2},s^{2}\right) ds_{2}\right) \left( \nabla u\left( x,t\right) +\nabla
_{y_{1}}u_{1}\left( x,t,y_{1}\right) \right.  \label{second local case 12} \\
+\left. \nabla _{y_{2}}u_{2}\left( x,t,y^{2},s_{1}\right) \right)
dy_{2}ds_{1}=0\text{.}  \notag
\end{gather}

\item In the case when $q>4+p$, the coefficient is characterized by%
\begin{gather*}
b\nabla u\left( x,t\right) =\int_{Y^{2}}\left( \int_{S^{2}}a\left(
y^{2},s^{2}\right) ds^{2}\right) \left( \nabla u\left( x,t\right) +\nabla
_{y_{1}}u_{1}\left( x,t,y_{1}\right) \right. \\
+\left. \nabla _{y_{2}}u_{2}\left( x,t,y^{2}\right) \right) dy^{2}
\end{gather*}%
and the local problems are given by%
\begin{gather}
-\nabla _{y_{2}}\cdot \left( \int_{S^{2}}a\left( y^{2},s^{2}\right)
ds^{2}\right) \left( \nabla u\left( x,t\right) +\nabla _{y_{1}}u_{1}\left(
x,t,y_{1}\right) \right.  \label{first local case 13} \\
+\left. \nabla _{y_{2}}u_{2}\left( x,t,y^{2}\right) \right) =0  \notag
\end{gather}%
and%
\begin{gather}
-\nabla _{y_{1}}\cdot \int_{Y_{2}}\left( \int_{S^{2}}a\left(
y^{2},s^{2}\right) ds^{2}\right) \left( \nabla u\left( x,t\right) +\nabla
_{y_{1}}u_{1}\left( x,t,y_{1}\right) \right.  \label{second local case 13} \\
+\left. \nabla _{y_{2}}u_{2}\left( x,t,y^{2}\right) \right) dy_{2}=0\text{,}
\notag
\end{gather}%
where $u_{1}\in L^{2}(\Omega _{T};H_{\sharp }^{1}(Y_{1})/%
\mathbb{R}
)$ and $u_{2}\in L^{2}(\Omega _{T}\times Y_{1};H_{\sharp }^{1}(Y_{2})/%
\mathbb{R}
)$.
\end{enumerate}
\end{theorem}

\begin{proof}
Since $\left\{ u_{\varepsilon }\right\} $ satisfies the a priori estimate (%
\ref{a priori}) and the conditions (\ref{villkor1-3}) and (\ref{villkor2-3}%
), Theorem \ref{Theorem - gradient charact} gives us (\ref{svag
homogenisering}), (\ref{3,3 homogenisering}) and (\ref{gradsplit
homogenisering}). The continuation of this proof will be divided into three
parts. We start by finding the homogenized problem (\ref{homogeniserat
problem}) followed by proving independencies of local time variables and
determining the local problems, which together will provide us with the
characterizations of the homogenized coefficient for all 13 cases.

Taking the test function%
\begin{equation*}
v\left( x\right) c\left( t\right) =v_{1}\left( x\right) c_{1}\left( t\right) 
\text{,}
\end{equation*}%
where $v_{1}\in H_{0}^{1}(\Omega )$ and $c_{1}\in \emph{D}(0,T)$, in the
weak form (\ref{weak form}) and letting $\varepsilon $ tend to zero, Theorem %
\ref{Theorem - gradient charact} yields%
\begin{gather*}
\int_{\Omega _{T}}\int_{\mathcal{Y}_{2,2}}a\left( y^{2},s^{2}\right) \left(
\nabla u\left( x,t\right) +\nabla _{y_{1}}u_{1}\left( x,t,y_{1},s^{2}\right)
+\nabla _{y_{2}}u_{2}\left( x,t,y^{2},s^{2}\right) \right) \\
\cdot \nabla v_{1}\left( x\right) c_{1}\left( t\right) \emph{d}y^{2}\emph{d}%
s^{2}\emph{d}x\emph{d}t=\int_{\Omega _{T}}f\left( x,t\right) v_{1}\left(
x\right) c_{1}\left( t\right) \emph{d}x\emph{d}t\text{.}
\end{gather*}%
By the Variational Lemma we arrive at

\begin{gather}
\int_{\Omega }\left( \int_{\mathcal{Y}_{2,2}}a\left( y^{2},s^{2}\right)
\left( u\left( x,t\right) +\nabla _{y_{1}}u_{1}\left( x,t,y_{1},s^{2}\right)
+\nabla _{y_{2}}u_{2}\left( x,t,y^{2},s^{2}\right) \right) \emph{d}y^{2}%
\emph{d}s^{2}\right)  \label{homogeniserade bevis} \\
\cdot \nabla v_{1}\left( x\right) \emph{d}x=\int_{\Omega }f\left( x,t\right)
v_{1}\left( x\right) \emph{d}x  \notag
\end{gather}%
a.e. in $(0,T)$, which is the weak form of (\ref{homogeniserat problem}).

We start by deriving a common ground, divided into two paths, for the
reasoning about independencies and the local problems. For the first path,
in the weak form (\ref{weak form}), we choose a test function which captures
the oscillations from the second microscopic scale $\varepsilon
_{2}=\varepsilon ^{2}$, more precise we choose%
\begin{equation}
v\left( x\right) c\left( t\right) =\varepsilon ^{k}v_{1}\left( x\right)
v_{2}\left( \frac{x}{\varepsilon }\right) v_{3}\left( \frac{x}{\varepsilon
^{2}}\right) c_{1}\left( t\right) c_{2}\left( \frac{t}{\varepsilon ^{q}}%
\right) c_{3}\left( \frac{t}{\varepsilon ^{r}}\right) \text{,}
\label{test functions 1}
\end{equation}%
where $k>0$, $v_{1}\in \emph{D}(\Omega )$, $v_{2}\in C_{\sharp }^{\infty
}(Y_{1})$, $v_{3}\in C_{\sharp }^{\infty }(Y_{2})/%
\mathbb{R}
$, $c_{1}\in \emph{D}(0,T)$,$\ c_{2}\in C_{\sharp }^{\infty }(S_{1})$ and $%
c_{3}\in C_{\sharp }^{\infty }(S_{2})$. After differentiations we arrive at%
\begin{gather*}
\int_{\Omega _{T}}-u_{\varepsilon }\left( x,t\right) v_{1}\left( x\right)
v_{2}\left( \frac{x}{\varepsilon }\right) v_{3}\left( \frac{x}{\varepsilon
^{2}}\right) \left( \varepsilon ^{k+p}\partial _{t}c_{1}\left( t\right)
c_{2}\left( \frac{t}{\varepsilon ^{q}}\right) c_{3}\left( \frac{t}{%
\varepsilon ^{r}}\right) \right. \\
+\left. \varepsilon ^{k+p-q}c_{1}\left( t\right) \partial
_{s_{1}}c_{2}\left( \frac{t}{\varepsilon ^{q}}\right) c_{3}\left( \frac{t}{%
\varepsilon ^{r}}\right) +\varepsilon ^{k+p-r}c_{1}\left( t\right)
c_{2}\left( \frac{t}{\varepsilon ^{q}}\right) \partial _{s_{2}}c_{3}\left( 
\frac{t}{\varepsilon ^{r}}\right) \right) \\
+a\left( \frac{x}{\varepsilon },\frac{x}{\varepsilon ^{2}},\frac{t}{%
\varepsilon ^{q}},\frac{t}{\varepsilon ^{r}}\right) \nabla u_{\varepsilon
}\left( x,t\right) \cdot \left( \varepsilon ^{k}\nabla v_{1}\left( x\right)
v_{2}\left( \frac{x}{\varepsilon }\right) v_{3}\left( \frac{x}{\varepsilon
^{2}}\right) \right. \\
+\left. \varepsilon ^{k-1}v_{1}\left( x\right) \nabla _{y_{1}}v_{2}\left( 
\frac{x}{\varepsilon }\right) v_{3}\left( \frac{x}{\varepsilon ^{2}}\right)
+\varepsilon ^{k-2}v_{1}\left( x\right) v_{2}\left( \frac{x}{\varepsilon }%
\right) \nabla _{y_{2}}v_{3}\left( \frac{x}{\varepsilon ^{2}}\right) \right)
\\
\times c_{1}\left( t\right) c_{2}\left( \frac{t}{\varepsilon ^{q}}\right)
c_{3}\left( \frac{t}{\varepsilon ^{r}}\right) \emph{d}x\emph{d}t \\
=\int_{\Omega _{T}}f\left( x,t\right) \varepsilon ^{k}v_{1}\left( x\right)
v_{2}\left( \frac{x}{\varepsilon }\right) v_{3}\left( \frac{x}{\varepsilon
^{2}}\right) c_{1}\left( t\right) c_{2}\left( \frac{t}{\varepsilon ^{q}}%
\right) c_{3}\left( \frac{t}{\varepsilon ^{r}}\right) \emph{d}x\emph{d}t%
\text{.}
\end{gather*}%
Passing to the limit, omitting terms that obviously tend to zero, we have%
\begin{gather}
\lim_{\varepsilon \rightarrow 0}\left( \int_{\Omega _{T}}-u_{\varepsilon
}\left( x,t\right) v_{1}\left( x\right) v_{2}\left( \frac{x}{\varepsilon }%
\right) v_{3}\left( \frac{x}{\varepsilon ^{2}}\right) \left( \varepsilon
^{k+p-q}c_{1}\left( t\right) \partial _{s_{1}}c_{2}\left( \frac{t}{%
\varepsilon ^{q}}\right) c_{3}\left( \frac{t}{\varepsilon ^{r}}\right)
\right. \right.  \notag \\
+\left. \varepsilon ^{k+p-r}c_{1}\left( t\right) c_{2}\left( \frac{t}{%
\varepsilon ^{q}}\right) \partial _{s_{2}}c_{3}\left( \frac{t}{\varepsilon
^{r}}\right) \right)  \label{point of departue 1} \\
+a\left( \frac{x}{\varepsilon },\frac{x}{\varepsilon ^{2}},\frac{t}{%
\varepsilon ^{q}},\frac{t}{\varepsilon ^{r}}\right) \nabla u_{\varepsilon
}\left( x,t\right) \cdot \left( \varepsilon ^{k-1}v_{1}\left( x\right)
\nabla _{y_{1}}v_{2}\left( \frac{x}{\varepsilon }\right) v_{3}\left( \frac{x%
}{\varepsilon ^{2}}\right) \right.  \notag \\
+\left. \left. \varepsilon ^{k-2}v_{1}\left( x\right) v_{2}\left( \frac{x}{%
\varepsilon }\right) \nabla _{y_{2}}v_{3}\left( \frac{x}{\varepsilon ^{2}}%
\right) \right) c_{1}\left( t\right) c_{2}\left( \frac{t}{\varepsilon ^{q}}%
\right) c_{3}\left( \frac{t}{\varepsilon ^{r}}\right) \emph{d}x\emph{d}%
t\right) =0\text{.}  \notag
\end{gather}%
For the second path, i.e. the one with respect to the first spatial
microscopic scale $\varepsilon _{1}=\varepsilon $, we let%
\begin{equation}
v\left( x\right) c\left( t\right) =\varepsilon ^{k}v_{1}\left( x\right)
v_{2}\left( \frac{x}{\varepsilon }\right) c_{1}\left( t\right) c_{2}\left( 
\frac{t}{\varepsilon ^{q}}\right) c_{3}\left( \frac{t}{\varepsilon ^{r}}%
\right) \text{,}  \label{test functions 2}
\end{equation}%
where $k>0$, $v_{1}\in \emph{D}(\Omega )$, $v_{2}\in C_{\sharp }^{\infty
}(Y_{1})/%
\mathbb{R}
$, $c_{1}\in \emph{D}(0,T)$,$\ c_{2}\in C_{\sharp }^{\infty }(S_{1})$ and $%
c_{3}\in C_{\sharp }^{\infty }(S_{2})$, act as a test function in the weak
form (\ref{weak form}). Differentiating leads to%
\begin{gather*}
\int_{\Omega _{T}}-u_{\varepsilon }\left( x,t\right) v_{1}\left( x\right)
v_{2}\left( \frac{x}{\varepsilon }\right) \left( \varepsilon ^{k+p}\partial
_{t}c_{1}\left( t\right) c_{2}\left( \frac{t}{\varepsilon ^{q}}\right)
c_{3}\left( \frac{t}{\varepsilon ^{r}}\right) \right. \\
+\left. \varepsilon ^{k+p-q}c_{1}\left( t\right) \partial
_{s_{1}}c_{2}\left( \frac{t}{\varepsilon ^{q}}\right) c_{3}\left( \frac{t}{%
\varepsilon ^{r}}\right) +\varepsilon ^{k+p-r}c_{1}\left( t\right)
c_{2}\left( \frac{t}{\varepsilon ^{q}}\right) \partial _{s_{2}}c_{3}\left( 
\frac{t}{\varepsilon ^{r}}\right) \right) \\
+a\left( \frac{x}{\varepsilon },\frac{x}{\varepsilon ^{2}},\frac{t}{%
\varepsilon ^{q}},\frac{t}{\varepsilon ^{r}}\right) \nabla u_{\varepsilon
}\left( x,t\right) \cdot \left( \varepsilon ^{k}\nabla v_{1}\left( x\right)
v_{2}\left( \frac{x}{\varepsilon }\right) +\varepsilon ^{k-1}v_{1}\left(
x\right) \nabla _{y_{1}}v_{2}\left( \frac{x}{\varepsilon }\right) \right) \\
\times c_{1}\left( t\right) c_{2}\left( \frac{t}{\varepsilon ^{q}}\right)
c_{3}\left( \frac{t}{\varepsilon ^{r}}\right) \emph{d}x\emph{d}t \\
=\int_{\Omega _{T}}f\left( x,t\right) \varepsilon ^{k}v_{1}\left( x\right)
v_{2}\left( \frac{x}{\varepsilon }\right) c_{1}\left( t\right) c_{2}\left( 
\frac{t}{\varepsilon ^{q}}\right) c_{3}\left( \frac{t}{\varepsilon ^{r}}%
\right) \emph{d}x\emph{d}t
\end{gather*}%
and as $\varepsilon \rightarrow 0$, after omitting terms that vanish, we have%
\begin{gather}
\lim_{\varepsilon \rightarrow 0}\left( \int_{\Omega _{T}}-u_{\varepsilon
}\left( x,t\right) v_{1}\left( x\right) v_{2}\left( \frac{x}{\varepsilon }%
\right) \left( \varepsilon ^{k+p-q}c_{1}\left( t\right) \partial
_{s_{1}}c_{2}\left( \frac{t}{\varepsilon ^{q}}\right) c_{3}\left( \frac{t}{%
\varepsilon ^{r}}\right) \right. \right.  \notag \\
+\left. \varepsilon ^{k+p-r}c_{1}\left( t\right) c_{2}\left( \frac{t}{%
\varepsilon ^{q}}\right) \partial _{s_{2}}c_{3}\left( \frac{t}{\varepsilon
^{r}}\right) \right)  \label{point of departue 2} \\
+a\left( \frac{x}{\varepsilon },\frac{x}{\varepsilon ^{2}},\frac{t}{%
\varepsilon ^{q}},\frac{t}{\varepsilon ^{r}}\right) \nabla u_{\varepsilon
}\left( x,t\right) \cdot \varepsilon ^{k-1}v_{1}\left( x\right) \nabla
_{y_{1}}v_{2}\left( \frac{x}{\varepsilon }\right)  \notag \\
\times \left. c_{1}\left( t\right) c_{2}\left( \frac{t}{\varepsilon ^{q}}%
\right) c_{3}\left( \frac{t}{\varepsilon ^{r}}\right) \emph{d}x\emph{d}%
t\right) =0\text{.}  \notag
\end{gather}

Now we are ready to prove the independencies of local time variables and we
start by showing when $u_{2}$ is independent of $s_{2}$. Let $r>4+p$ and
choose $k=r-p-2$ in (\ref{test functions 1}). As $\varepsilon \rightarrow 0$%
, applying Theorems \ref{Theorem - gradient charact} and \ref{Theorem -
Compactresult vw}, the limit of (\ref{point of departue 1}) becomes%
\begin{gather*}
\int_{\Omega _{T}}\int_{\mathcal{Y}_{2,2}}-u_{2}\left(
x,t,y^{2},s^{2}\right) v_{1}\left( x\right) v_{2}\left( y_{1}\right)
v_{3}\left( y_{2}\right) \\
\times c_{1}\left( t\right) c_{2}\left( s_{1}\right) \partial
_{s_{2}}c_{3}\left( s_{2}\right) \emph{d}y^{2}\emph{d}s^{2}\emph{d}x\emph{d}%
t=0
\end{gather*}%
and by the Variational Lemma%
\begin{equation*}
\int_{S_{2}}-u_{2}\left( x,t,y^{2},s^{2}\right) \partial _{s_{2}}c_{3}\left(
s_{2}\right) \emph{d}s_{2}=0
\end{equation*}%
a.e. in $\Omega _{T}\times \mathcal{Y}_{2,1}$, which indicates that $u_{2}$
is independent of $s_{2}$.

Now we show independence of $s_{1}$ in $u_{2}$. Let $q>4+p$ and since $r>q$
this implies that $u_{2}$ is independent of $s_{2}$. Therefore we let $%
c_{3}\equiv 1$ in (\ref{test functions 1}) and we choose $k=q-p-2$. Passing
to the limit in (\ref{point of departue 1}), Theorems \ref{Theorem -
gradient charact} and \ref{Theorem - Compactresult vw} yield%
\begin{equation*}
\int_{\Omega _{T}}\int_{\mathcal{Y}_{2,2}}-u_{2}\left(
x,t,y^{2},s_{1}\right) v_{1}\left( x\right) v_{2}\left( y_{1}\right)
v_{3}\left( y_{2}\right) c_{1}\left( t\right) \partial _{s_{1}}c_{2}\left(
s_{1}\right) \emph{d}y^{2}\emph{d}s^{2}\emph{d}x\emph{d}t=0
\end{equation*}%
and integrating over $S_{2}$ and applying the Variational Lemma on $\Omega
_{T}\times Y^{2}$, we obtain that $u_{2}$ is independent of $s_{1}$.

Next we show independence of $s_{2}$ in $u_{1}$. Let $r>2+p$ and choose $%
k=r-p-1$ in (\ref{test functions 2}). Letting $\varepsilon $ tend to zero in
(\ref{point of departue 2}), applying Theorems \ref{Theorem - gradient
charact} and \ref{Theorem - Compactresult vw}, we have%
\begin{equation*}
\int_{\Omega _{T}}\int_{\mathcal{Y}_{1,2}}-u_{1}\left(
x,t,y_{1},s^{2}\right) v_{1}\left( x\right) v_{2}\left( y_{1}\right)
c_{1}\left( t\right) c_{2}\left( s_{1}\right) \partial _{s_{2}}c_{3}\left(
s_{2}\right) \emph{d}y_{1}\emph{d}s^{2}\emph{d}x\emph{d}t=0
\end{equation*}%
and the Variational Lemma on $\Omega _{T}\times \mathcal{Y}_{1,1}$ shows
that $u_{1}$ is independent of $s_{2}$.

The last independence to show is when $u_{1}$ is independent of $s_{1}$.
Here we let $q>2+p$ and recalling that since $r>q$, $u_{1}$ is independent
of $s_{2}$. In (\ref{test functions 2}) we choose $k=q-p-1$ and set $%
c_{3}\equiv 1$. As $\varepsilon \rightarrow 0$ in (\ref{point of departue 2}%
), Theorems \ref{Theorem - gradient charact} and \ref{Theorem -
Compactresult vw} give%
\begin{equation*}
\int_{\Omega _{T}}\int_{\mathcal{Y}_{1,2}}-u_{1}\left(
x,t,y_{1},s_{1}\right) v_{1}\left( x\right) v_{2}\left( y_{1}\right)
c_{1}\left( t\right) \partial _{s_{1}}c_{2}\left( s_{1}\right) \emph{d}y_{1}%
\emph{d}s^{2}\emph{d}x\emph{d}t=0\text{.}
\end{equation*}%
Integrating over $S_{2}$ and using the Variational Lemma on $\Omega
_{T}\times Y_{1}$ we have that $u_{1}$ is independent of $s_{1}$.

To sum up, we know that $u_{1}$ is independent of $s_{2}$ whenever $r>2+p$
and that $u_{1}$ is independent of both $s_{1}$ and $s_{2}$, when $q>2+p$.
In the case when $r>4+p$, $u_{2}$ (and of course also $u_{1}$) is
independent of $s_{2}$ and if $q>4+p$ we have that $u_{2}$ (and $u_{1}$) is
independent of both $s_{1}$ and $s_{2}$. These independencies together with (%
\ref{homogeniserade bevis}) give the formulas for the homogenized
coefficient in the cases 1-13.

Now we are going to derive the system of local problems for each of the 13
cases. Each case has a system consisting of two local problems. The first
local problem is with respect to the faster microscopic scale $\varepsilon
_{2}=\varepsilon ^{2}$ and our point of departure is always (\ref{point of
departue 1}) where we have chosen $k=2$ in (\ref{test functions 1}). The
second local problem is with respect to the slower microscopic scale $%
\varepsilon _{1}=\varepsilon $ and the point of departure here is (\ref%
{point of departue 2}) where we have taken $k=1$ in (\ref{test functions 2}).

\textit{Case 1: }$r<2+p$. To obtain the first local problem we let $%
\varepsilon \rightarrow 0$ in (\ref{point of departue 1}) and applying
Theorem \ref{Theorem - gradient charact} we have%
\begin{gather*}
\int_{\Omega _{T}}\int_{\mathcal{Y}_{2,2}}a\left( y^{2},s^{2}\right) \left(
\nabla u\left( x,t\right) +\nabla _{y_{1}}u_{1}\left( x,t,y_{1},s^{2}\right)
+\nabla _{y_{2}}u_{2}\left( x,t,y^{2},s^{2}\right) \right) \\
\times v_{1}\left( x\right) v_{2}\left( y_{1}\right) \cdot \nabla
_{y_{2}}v_{3}\left( y_{2}\right) c_{1}\left( t\right) c_{2}\left(
s_{1}\right) c_{3}\left( s_{2}\right) \emph{d}y^{2}\emph{d}s^{2}\emph{d}x%
\emph{d}t=0\text{.}
\end{gather*}%
By the Variational Lemma on $\Omega _{T}\times \mathcal{Y}_{1,2}$, we obtain
the weak form of (\ref{first local case 1}).

For the second local problem, passing to the limit in (\ref{point of
departue 2}), using Theorems \ref{Theorem - gradient charact} and \ref%
{Theorem - Compactresult vw}, we obtain%
\begin{gather*}
\int_{\Omega _{T}}\int_{\mathcal{Y}_{2,2}}a\left( y^{2},s^{2}\right) \left(
\nabla u\left( x,t\right) +\nabla _{y_{1}}u_{1}\left( x,t,y_{1},s^{2}\right)
+\nabla _{y_{2}}u_{2}\left( x,t,y^{2},s^{2}\right) \right) \\
\times v_{1}\left( x\right) \cdot \nabla _{y_{1}}v_{2}\left( y_{1}\right)
c_{1}\left( t\right) c_{2}\left( s_{1}\right) c_{3}\left( s_{2}\right) \emph{%
d}y^{2}\emph{d}s^{2}\emph{d}x\emph{d}t=0
\end{gather*}%
and the Variational Lemma on $\Omega _{T}\times S^{2}$ gives us the weak
form of (\ref{second local case 1}).

\textit{Case 2:} $r=2+p$. Passing to the limit in (\ref{point of departue 1}%
) yields the same result as for the first local problem in case 1, which is
the weak form of (\ref{first local case 2}).

For the second local problem, we apply Theorems \ref{Theorem - gradient
charact} and \ref{Theorem - Compactresult vw} as we pass to the limit in (%
\ref{point of departue 2}) to get%
\begin{gather*}
\int_{\Omega _{T}}\int_{\mathcal{Y}_{1,2}}-u_{1}\left(
x,t,y_{1},s^{2}\right) v_{1}\left( x\right) v_{2}\left( y_{1}\right)
c_{1}\left( t\right) c_{2}\left( s_{1}\right) \partial _{s_{2}}c_{3}\left(
s_{2}\right) \emph{d}y_{1}\emph{d}s^{2}\emph{d}x\emph{d}t \\
+\int_{\Omega _{T}}\int_{\mathcal{Y}_{2,2}}a\left( y^{2},s^{2}\right) \left(
\nabla u\left( x,t\right) +\nabla _{y_{1}}u_{1}\left( x,t,y_{1},s^{2}\right)
+\nabla _{y_{2}}u_{2}\left( x,t,y^{2},s^{2}\right) \right) \\
\times v_{1}\left( x\right) \cdot \nabla _{y_{1}}v_{2}\left( y_{1}\right)
c_{1}\left( t\right) c_{2}\left( s_{1}\right) c_{3}\left( s_{2}\right) \emph{%
d}y^{2}\emph{d}s^{2}\emph{d}x\emph{d}t=0\text{.}
\end{gather*}%
Using the Variational Lemma on $\Omega _{T}\times S_{1}$, we get the weak
form of (\ref{second local case 2}).

\textit{Case 3:} $2+p<r<4+p$ and $q<2+p$. Passing to the limit in (\ref%
{point of departue 1}) and applying Theorems \ref{Theorem - gradient charact}
and \ref{Theorem - Compactresult vw}, recalling that $u_{1}$ is independent
of $s_{2}$, we arrive at%
\begin{gather*}
\int_{\Omega _{T}}\int_{\mathcal{Y}_{2,2}}a\left( y^{2},s^{2}\right) \left(
\nabla u\left( x,t\right) +\nabla _{y_{1}}u_{1}\left( x,t,y_{1},s_{1}\right)
+\nabla _{y_{2}}u_{2}\left( x,t,y^{2},s^{2}\right) \right) \\
\times v_{1}\left( x\right) v_{2}\left( y_{1}\right) \cdot \nabla
_{y_{2}}v_{3}\left( y_{2}\right) c_{1}\left( t\right) c_{2}\left(
s_{1}\right) c_{3}\left( s_{2}\right) \emph{d}y^{2}\emph{d}s^{2}\emph{d}x%
\emph{d}t=0\text{.}
\end{gather*}%
Applying the Variational Lemma on $\Omega _{T}\times \mathcal{Y}_{1,2}$ we
have the weak form of (\ref{first local case 3}).

Because of the independence of $s_{2}$ in $u_{1}$, we can let $c_{3}\equiv 1$
in (\ref{test functions 2}). As $\varepsilon \rightarrow 0$ in (\ref{point
of departue 2}), by Theorems \ref{Theorem - gradient charact} and \ref%
{Theorem - Compactresult vw} we obtain%
\begin{gather*}
\int_{\Omega _{T}}\int_{\mathcal{Y}_{2,2}}a\left( y^{2},s^{2}\right) \left(
\nabla u\left( x,t\right) +\nabla _{y_{1}}u_{1}\left( x,t,y_{1},s_{1}\right)
+\nabla _{y_{2}}u_{2}\left( x,t,y^{2},s^{2}\right) \right) \\
\times v_{1}\left( x\right) \cdot \nabla _{y_{1}}v_{2}\left( y_{1}\right)
c_{1}\left( t\right) c_{2}\left( s_{1}\right) \emph{d}y^{2}\emph{d}s^{2}%
\emph{d}x\emph{d}t=0
\end{gather*}%
and the Variational Lemma on $\Omega _{T}\times S_{1}$ gives the weak form
of (\ref{second local case 3}).

\textit{Case 4:} $r<4+p$ and $q=2+p$. Passing to the limit in (\ref{point of
departue 1}), remembering that $u_{1}$ is independent of $s_{2}$, by
Theorems \ref{Theorem - gradient charact} and \ref{Theorem - Compactresult
vw} we arrive at the same local problem as the first one in case 3, which is
the weak form of (\ref{first local case 4}).

Letting $c_{3}\equiv 1$ in (\ref{test functions 2}) and passing to the limit
in (\ref{point of departue 2}), applying Theorems \ref{Theorem - gradient
charact} and \ref{Theorem - Compactresult vw}, we get%
\begin{gather*}
\int_{\Omega _{T}}\int_{\mathcal{Y}_{1,2}}-u_{1}\left(
x,t,y_{1},s_{1}\right) v_{1}\left( x\right) v_{2}\left( y_{1}\right)
c_{1}\left( t\right) \partial _{s_{1}}c_{2}\left( s_{1}\right) \emph{d}y_{1}%
\emph{d}s^{2}\emph{d}x\emph{d}t \\
+\int_{\Omega _{T}}\int_{\mathcal{Y}_{2,2}}a\left( y^{2},s^{2}\right) \left(
\nabla u\left( x,t\right) +\nabla _{y_{1}}u_{1}\left( x,t,y_{1},s_{1}\right)
+\nabla _{y_{2}}u_{2}\left( x,t,y^{2},s^{2}\right) \right) \\
\times v_{1}\left( x\right) \cdot \nabla _{y_{1}}v_{2}\left( y_{1}\right)
c_{1}\left( t\right) c_{2}\left( s_{1}\right) \emph{d}y^{2}\emph{d}s^{2}%
\emph{d}x\emph{d}t=0\text{.}
\end{gather*}%
Integrating over $S_{2}$ in the first integral and applying the Variational
Lemma on $\Omega _{T}$ we get the weak form of (\ref{second local case 4}).

\textit{Case 5:} $r<4+p$ and $q>2+p$. Remembering that $u_{1}$ is
independent of both $s_{1}$ and $s_{2}$, when $\varepsilon \rightarrow 0$ in
(\ref{point of departue 1}) we apply Theorems \ref{Theorem - gradient
charact} and \ref{Theorem - Compactresult vw} and have%
\begin{gather*}
\int_{\Omega _{T}}\int_{\mathcal{Y}_{2,2}}a\left( y^{2},s^{2}\right) \left(
\nabla u\left( x,t\right) +\nabla _{y_{1}}u_{1}\left( x,t,y_{1}\right)
+\nabla _{y_{2}}u_{2}\left( x,t,y^{2},s^{2}\right) \right) \\
\times v_{1}\left( x\right) v_{2}\left( y_{1}\right) \cdot \nabla
_{y_{2}}v_{3}\left( y_{2}\right) c_{1}\left( t\right) c_{2}\left(
s_{1}\right) c_{3}\left( s_{2}\right) \emph{d}y^{2}\emph{d}s^{2}\emph{d}x%
\emph{d}t=0\text{.}
\end{gather*}%
By using the Variational Lemma on $\Omega _{T}\times \mathcal{Y}_{1,2}$ we
arrive at the weak form of (\ref{first local case 5}).

Because of the independencies, we can let $c_{2}\equiv 1$ and $c_{3}\equiv 1$
in (\ref{test functions 2}). Applying Theorem \ref{Theorem - gradient
charact} as $\varepsilon $ tends to zero in (\ref{point of departue 2})
yields%
\begin{gather*}
\int_{\Omega _{T}}\int_{\mathcal{Y}_{2,2}}a\left( y^{2},s^{2}\right) \left(
\nabla u\left( x,t\right) +\nabla _{y_{1}}u_{1}\left( x,t,y_{1}\right)
+\nabla _{y_{2}}u_{2}\left( x,t,y^{2},s^{2}\right) \right) \\
\times v_{1}\left( x\right) \cdot \nabla _{y_{1}}v_{2}\left( y_{1}\right)
c_{1}\left( t\right) \emph{d}y^{2}\emph{d}s^{2}\emph{d}x\emph{d}t=0
\end{gather*}%
and by the Variational Lemma on $\Omega _{T}$ we get the weak form of (\ref%
{second local case 5}).

\textit{Case 6:} $r=4+p$ and $q<2+p$. Noting that $u_{1}$ is independent of $%
s_{2}$, passing to the limit in (\ref{point of departue 1}), Theorems \ref%
{Theorem - gradient charact} and \ref{Theorem - Compactresult vw} give us%
\begin{gather*}
\int_{\Omega _{T}}\int_{\mathcal{Y}_{2,2}}-u_{2}\left(
x,t,y^{2},s^{2}\right) v_{1}\left( x\right) v_{2}\left( y_{1}\right)
v_{3}\left( y_{2}\right) c_{1}\left( t\right) c_{2}\left( s_{1}\right)
\partial _{s_{2}}c_{3}\left( s_{2}\right) \emph{d}y^{2}\emph{d}s^{2}\emph{d}x%
\emph{d}t \\
+\int_{\Omega _{T}}\int_{\mathcal{Y}_{2,2}}a\left( y^{2},s^{2}\right) \left(
\nabla u\left( x,t\right) +\nabla _{y_{1}}u_{1}\left( x,t,y_{1},s_{1}\right)
+\nabla _{y_{2}}u_{2}\left( x,t,y^{2},s^{2}\right) \right) \\
\times v_{1}\left( x\right) v_{2}\left( y_{1}\right) \cdot \nabla
_{y_{2}}v_{3}\left( y_{2}\right) c_{1}\left( t\right) c_{2}\left(
s_{1}\right) c_{3}\left( s_{2}\right) \emph{d}y^{2}\emph{d}s^{2}\emph{d}x%
\emph{d}t=0\text{.}
\end{gather*}%
Applying the Variational Lemma on $\Omega _{T}\times \mathcal{Y}_{1,1}$ we
have the weak form of (\ref{first local case 6}).

Because of the independence in $u_{1}$, we can let $c_{3}\equiv 1$ in (\ref%
{test functions 2}) and as $\varepsilon \rightarrow 0$ (\ref{point of
departue 2}) becomes, due to Theorems \ref{Theorem - gradient charact} and %
\ref{Theorem - Compactresult vw}, 
\begin{gather*}
\int_{\Omega _{T}}\int_{\mathcal{Y}_{2,2}}a\left( y^{2},s^{2}\right) \left(
\nabla u\left( x,t\right) +\nabla _{y_{1}}u_{1}\left( x,t,y_{1},s_{1}\right)
+\nabla _{y_{2}}u_{2}\left( x,t,y^{2},s^{2}\right) \right) \\
\times v_{1}\left( x\right) \cdot \nabla _{y_{1}}v_{2}\left( y_{1}\right)
c_{1}\left( t\right) c_{2}\left( s_{1}\right) \emph{d}y^{2}\emph{d}s^{2}%
\emph{d}x\emph{d}t=0\text{.}
\end{gather*}%
Using the Variational Lemma on $\Omega _{T}\times S_{1}$ we obtain the weak
form of (\ref{second local case 6}).

\textit{Case 7:} $r=4+p$ and $q=2+p$. As $\varepsilon \rightarrow 0$ in (\ref%
{point of departue 1}), we end up with the same local problem as the first
one in case 6, which is the weak form of (\ref{first local case 7}).

Letting $\varepsilon $ tend to zero in (\ref{point of departue 2}),
recalling that $u_{1}$ is independent of $s_{2}$ so that $c_{3}\equiv 1$,
Theorems \ref{Theorem - gradient charact} and \ref{Theorem - Compactresult
vw} yield%
\begin{gather*}
\int_{\Omega _{T}}\int_{\mathcal{Y}_{1,2}}-u_{1}\left(
x,t,y_{1},s_{1}\right) v_{1}\left( x\right) v_{2}\left( y_{1}\right)
c_{1}\left( t\right) \partial _{s_{1}}c_{2}\left( s_{1}\right) \emph{d}y_{1}%
\emph{d}s^{2}\emph{d}x\emph{d}t \\
+\int_{\Omega _{T}}\int_{\mathcal{Y}_{2,2}}a\left( y^{2},s^{2}\right) \left(
\nabla u\left( x,t\right) +\nabla _{y_{1}}u_{1}\left( x,t,y_{1},s_{1}\right)
+\nabla _{y_{2}}u_{2}\left( x,t,y^{2},s^{2}\right) \right) \\
\times v_{1}\left( x\right) \cdot \nabla _{y_{1}}v_{2}\left( y_{1}\right)
c_{1}\left( t\right) c_{2}\left( s_{1}\right) \emph{d}y^{2}\emph{d}s^{2}%
\emph{d}x\emph{d}t=0\text{.}
\end{gather*}%
Integrating over $S_{2}$ in the first integral and taking the Variational
Lemma on $\Omega _{T}$ gives us the weak form of (\ref{second local case 7}).

\textit{Case 8:} $r=4+p$ and $q>2+p$. Letting $\varepsilon $ tend to zero in
(\ref{point of departue 1}), observing that $u_{1}$ is independent of both $%
s_{1}$ and $s_{2}$, Theorems \ref{Theorem - gradient charact} and \ref%
{Theorem - Compactresult vw} give%
\begin{gather*}
\int_{\Omega _{T}}\int_{\mathcal{Y}_{2,2}}-u_{2}\left(
x,t,y^{2},s^{2}\right) v_{1}\left( x\right) v_{2}\left( y_{1}\right)
v_{3}\left( y_{2}\right) c_{1}\left( t\right) c_{2}\left( s_{1}\right)
\partial _{s_{2}}c_{3}\left( s_{2}\right) \emph{d}y^{2}\emph{d}s^{2}\emph{d}x%
\emph{d}t \\
+\int_{\Omega _{T}}\int_{\mathcal{Y}_{2,2}}a\left( y^{2},s^{2}\right) \left(
\nabla u\left( x,t\right) +\nabla _{y_{1}}u_{1}\left( x,t,y_{1}\right)
+\nabla _{y_{2}}u_{2}\left( x,t,y^{2},s^{2}\right) \right) \\
\times v_{1}\left( x\right) v_{2}\left( y_{1}\right) \cdot \nabla
_{y_{2}}v_{3}\left( y_{2}\right) c_{1}\left( t\right) c_{2}\left(
s_{1}\right) c_{3}\left( s_{2}\right) \emph{d}y^{2}\emph{d}s^{2}\emph{d}x%
\emph{d}t=0
\end{gather*}%
and by applying the Variational Lemma on $\Omega _{T}\times \mathcal{Y}%
_{1,1} $ we get the weak form of (\ref{first local case 8}).

For the second local problem, due to independencies in $u_{1}$, we can let
both $c_{2}\equiv 1$ and $c_{3}\equiv 1$ in (\ref{test functions 2}).
Letting $\varepsilon \rightarrow 0$ in (\ref{point of departue 2}), from
Theorem \ref{Theorem - gradient charact} we obtain%
\begin{gather*}
\int_{\Omega _{T}}\int_{\mathcal{Y}_{2,2}}a\left( y^{2},s^{2}\right) \left(
\nabla u\left( x,t\right) +\nabla _{y_{1}}u_{1}\left( x,t,y_{1}\right)
+\nabla _{y_{2}}u_{2}\left( x,t,y^{2},s^{2}\right) \right) \\
\times v_{1}\left( x\right) \cdot \nabla _{y_{1}}v_{2}\left( y_{1}\right)
c_{1}\left( t\right) \emph{d}y^{2}\emph{d}s^{2}\emph{d}x\emph{d}t=0
\end{gather*}%
and the Variational Lemma on $\Omega _{T}$ gives us the weak form of (\ref%
{second local case 8}).

\textit{Case 9:} $r>4+p$ and $q<2+p$. Recalling that $u_{2}$ (and $u_{1}$)
is independent of $s_{2}$, we can let $c_{3}\equiv 1$ in (\ref{test
functions 1}). Passing to the limit in (\ref{point of departue 1}), Theorem %
\ref{Theorem - gradient charact} gives us%
\begin{gather*}
\int_{\Omega _{T}}\int_{\mathcal{Y}_{2,2}}a\left( y^{2},s^{2}\right) \left(
\nabla u\left( x,t\right) +\nabla _{y_{1}}u_{1}\left( x,t,y_{1},s_{1}\right)
+\nabla _{y_{2}}u_{2}\left( x,t,y^{2},s_{1}\right) \right) \\
\times v_{1}\left( x\right) v_{2}\left( y_{1}\right) \cdot \nabla
_{y_{2}}v_{3}\left( y_{2}\right) c_{1}\left( t\right) c_{2}\left(
s_{1}\right) \emph{d}y^{2}\emph{d}s^{2}\emph{d}x\emph{d}t=0
\end{gather*}%
and using the Variational Lemma on $\Omega _{T}\times \mathcal{Y}_{1,1}$ we
obtain the weak form of (\ref{first local case 9}).

Due to the independence in $u_{1}$ we can let $c_{3}\equiv 1$ in (\ref{test
functions 2}) and as $\varepsilon \rightarrow 0$ in (\ref{point of departue
2}), Theorems \ref{Theorem - gradient charact} and \ref{Theorem -
Compactresult vw} yield%
\begin{gather*}
\int_{\Omega _{T}}\int_{\mathcal{Y}_{2,2}}a\left( y^{2},s^{2}\right) \left(
\nabla u\left( x,t\right) +\nabla _{y_{1}}u_{1}\left( x,t,y_{1},s_{1}\right)
+\nabla _{y_{2}}u_{2}\left( x,t,y^{2},s_{1}\right) \right) \\
\times v_{1}\left( x\right) \cdot \nabla _{y_{1}}v_{2}\left( y_{1}\right)
c_{1}\left( t\right) c_{2}\left( s_{1}\right) \emph{d}y^{2}\emph{d}s^{2}%
\emph{d}x\emph{d}t=0\text{.}
\end{gather*}%
By the Variational Lemma on $\Omega _{T}\times S_{1}$ we have the weak form
of (\ref{second local case 9}).

\textit{Case 10:} $r>4+p$ and $q=2+p$. Because of the independence of $s_{2}$
in $u_{2}$ we let $c_{3}\equiv 1$ in (\ref{test functions 1}) and as $%
\varepsilon $ tends to zero in (\ref{point of departue 1}), recalling that
also $u_{1}$ is independent of $s_{2}$, Theorems \ref{Theorem - gradient
charact} and \ref{Theorem - Compactresult vw} give the same first local
problem as in case 9, which is the weak form of (\ref{first local case 10}).

Again we can let $c_{3}\equiv 1$ in (\ref{test functions 2}), due to
independence in $u_{1}$. Letting $\varepsilon \rightarrow 0$ in (\ref{point
of departue 2}), from Theorems \ref{Theorem - gradient charact} and \ref%
{Theorem - Compactresult vw} we have%
\begin{gather*}
\int_{\Omega _{T}}\int_{\mathcal{Y}_{1,2}}-u_{1}\left(
x,t,y_{1},s_{1}\right) v_{1}\left( x\right) v_{2}\left( y_{1}\right)
c_{1}\left( t\right) \partial _{s_{1}}c_{2}\left( s_{1}\right) \emph{d}y_{1}%
\emph{d}s^{2}\emph{d}x\emph{d}t \\
+\int_{\Omega _{T}}\int_{\mathcal{Y}_{2,2}}a\left( y^{2},s^{2}\right) \left(
\nabla u\left( x,t\right) +\nabla _{y_{1}}u_{1}\left( x,t,y_{1},s_{1}\right)
+\nabla _{y_{2}}u_{2}\left( x,t,y^{2},s_{1}\right) \right) \\
\times v_{1}\left( x\right) \cdot \nabla _{y_{1}}v_{2}\left( y_{1}\right)
c_{1}\left( t\right) c_{2}\left( s_{1}\right) \emph{d}y^{2}\emph{d}s^{2}%
\emph{d}x\emph{d}t=0\text{.}
\end{gather*}%
Integrating over $S_{2}$ in the first integral and using the Variational
Lemma on $\Omega _{T}$ we get the weak form of (\ref{second local case 10}).

\textit{Case 11:} $r>4+p$ and $2+p<q<4+p$. Since $u_{2}$ is independent of $%
s_{2}$, we let $c_{3}\equiv 1$ in (\ref{test functions 1}). We also have
independence of $s_{1}$ and $s_{2}$ in $u_{1}$, so as $\varepsilon
\rightarrow 0$ in (\ref{point of departue 1}), Theorems \ref{Theorem -
gradient charact} and \ref{Theorem - Compactresult vw} give%
\begin{gather*}
\int_{\Omega _{T}}\int_{\mathcal{Y}_{2,2}}a\left( y^{2},s^{2}\right) \left(
\nabla u\left( x,t\right) +\nabla _{y_{1}}u_{1}\left( x,t,y_{1}\right)
+\nabla _{y_{2}}u_{2}\left( x,t,y^{2},s_{1}\right) \right) \\
\times v_{1}\left( x\right) v_{2}\left( y_{1}\right) \cdot \nabla
_{y_{2}}v_{3}\left( y_{2}\right) c_{1}\left( t\right) c_{2}\left(
s_{1}\right) \emph{d}y^{2}\emph{d}s^{2}\emph{d}x\emph{d}t=0\text{.}
\end{gather*}%
Applying the Variational Lemma on $\Omega _{T}\times \mathcal{Y}_{1,1}$ we
get the weak form of (\ref{first local case 11}).

Because of the independencies in $u_{1}$, for the second local problem, we
can let both $c_{2}\equiv 1$ and $c_{3}\equiv 1$ in (\ref{test functions 2}%
). Passing to the limit in (\ref{point of departue 2}), applying Theorem \ref%
{Theorem - gradient charact}, we end up with%
\begin{gather*}
\int_{\Omega _{T}}\int_{\mathcal{Y}_{2,2}}a\left( y^{2},s^{2}\right) \left(
\nabla u\left( x,t\right) +\nabla _{y_{1}}u_{1}\left( x,t,y_{1}\right)
+\nabla _{y_{2}}u_{2}\left( x,t,y^{2},s_{1}\right) \right) \\
\times v_{1}\left( x\right) \cdot \nabla _{y_{1}}v_{2}\left( y_{1}\right)
c_{1}\left( t\right) \emph{d}y^{2}\emph{d}s^{2}\emph{d}x\emph{d}t=0
\end{gather*}%
and from the Variational Lemma on $\Omega _{T}$ we obtain the weak form of (%
\ref{second local case 11}).

\textit{Case 12:} $q=4+p$. Since $u_{2}$ is independent of $s_{2}$ we can
take $c_{3}\equiv 1$ in (\ref{test functions 1}). Recalling that $u_{1}$ is
independent of $s_{1}$ and $s_{2}$, passing to the limit in (\ref{point of
departue 1}), from Theorems \ref{Theorem - gradient charact} and \ref%
{Theorem - Compactresult vw} we have%
\begin{gather*}
\int_{\Omega _{T}}\int_{\mathcal{Y}_{2,2}}-u_{2}\left(
x,t,y^{2},s_{1}\right) v_{1}\left( x\right) v_{2}\left( y_{1}\right)
v_{3}\left( y_{2}\right) c_{1}\left( t\right) \partial _{s_{1}}c_{2}\left(
s_{1}\right) \emph{d}y^{2}\emph{d}s^{2}\emph{d}x\emph{d}t \\
+\int_{\Omega _{T}}\int_{\mathcal{Y}_{2,2}}a\left( y^{2},s^{2}\right) \left(
\nabla u\left( x,t\right) +\nabla _{y_{1}}u_{1}\left( x,t,y_{1}\right)
+\nabla _{y_{2}}u_{2}\left( x,t,y^{2},s_{1}\right) \right) \\
\times v_{1}\left( x\right) v_{2}\left( y_{1}\right) \cdot \nabla
_{y_{2}}v_{3}\left( y_{2}\right) c_{1}\left( t\right) c_{2}\left(
s_{1}\right) \emph{d}y^{2}\emph{d}s^{2}\emph{d}x\emph{d}t=0\text{.}
\end{gather*}%
Integrating over $S_{2}$ in the first integral and applying the Variational
Lemma on $\Omega _{T}\times Y_{1}$ we have the weak form of (\ref{first
local case 12}).

Because of the independencies in $u_{1}$ we can let $c_{2}\equiv 1$ and $%
c_{3}\equiv 1$ in (\ref{test functions 2}) and as $\varepsilon $ tends to
zero in (\ref{point of departue 2}), we get the same result as for the
second local problem in case 11, sharing the weak form of (\ref{second local
case 12}).

\textit{Case 13:} $q>4+p$. Recalling that $u_{2}$ is independent of $s_{1}$
and $s_{2}$, we can set $c_{2}\equiv 1$ and $c_{3}\equiv 1$ in (\ref{test
functions 1}). Noting that also $u_{1}$ is independent of both $s_{1}$ and $%
s_{2}$, letting $\varepsilon \rightarrow 0$ in (\ref{point of departue 1}),
Theorem \ref{Theorem - gradient charact} yields%
\begin{gather*}
\int_{\Omega _{T}}\int_{\mathcal{Y}_{2,2}}a\left( y^{2},s^{2}\right) \left(
\nabla u\left( x,t\right) +\nabla _{y_{1}}u_{1}\left( x,t,y_{1}\right)
+\nabla _{y_{2}}u_{2}\left( x,t,y^{2}\right) \right) \\
\times v_{1}\left( x\right) v_{2}\left( y_{1}\right) \cdot \nabla
_{y_{2}}v_{3}\left( y_{2}\right) c_{1}\left( t\right) \emph{d}y^{2}\emph{d}%
s^{2}\emph{d}x\emph{d}t=0
\end{gather*}%
and applying the Variational Lemma on $\Omega _{T}\times Y_{1}$ gives the
weak form of (\ref{first local case 13}).

For the second local problem, we again let $c_{2}\equiv 1$ and $c_{3}\equiv 1
$ in (\ref{test functions 2}) and as $\varepsilon \rightarrow 0$ in (\ref%
{point of departue 2}), Theorem \ref{Theorem - gradient charact} gives%
\begin{gather*}
\int_{\Omega _{T}}\int_{\mathcal{Y}_{2,2}}a\left( y^{2},s^{2}\right) \left(
\nabla u\left( x,t\right) +\nabla _{y_{1}}u_{1}\left( x,t,y_{1}\right)
+\nabla _{y_{2}}u_{2}\left( x,t,y^{2}\right) \right)  \\
\times v_{1}\left( x\right) \cdot \nabla _{y_{1}}v_{2}\left( y_{1}\right)
c_{1}\left( t\right) \emph{d}y^{2}\emph{d}s^{2}\emph{d}x\emph{d}t=0\text{.}
\end{gather*}%
From the Variational Lemma on $\Omega _{T}$ we get the weak form of (\ref%
{second local case 13}).
\end{proof}


\begin{thebibliography}{99}
\bibitem{All} \textit{G. Allaire}: Homogenization and two-scale convergence.
SIAM J. Math. Anal. \textit{23} (1992), no. 6, 1482--1518. DOI
10.1137/0523084, Zbl 0770.35005, MR1185639

\bibitem{AlBr} \textit{G. Allaire, M. Briane}: Multiscale convergence and
reiterated homogenization. Proc. Roy. Soc. Edinburgh Sect. A. \textit{126}
(1996), no. 2, 297--342. DOI 10.1017/S0308210500022757, Zbl 0866.35017,
MR1386865

\bibitem{AlPi} \textit{G. Allaire, A. Piatnitski}: Homogenization of
nonlinear reaction-diffusion equation with a large reaction term. Ann. Univ.
Ferrara Sez. VII Sci. Mat. \textit{56} (2010), no. 1, 141--161. DOI
10.1007/s11565-010-0095-z, Zbl 1205.35019, MR2646529

\bibitem{BLP} \textit{A. Bensoussan, J.-L. Lions, G. Papanicolaou}:
Asymptotic analysis for periodic structures. Studies in Mathematics and its
Applications 5, North-Holland Publishing Co., Amsterdam-New York, 1978. DOI
10.1090/chel/374, Zbl 0404.35001, MR0503330

\bibitem{JoLoarxiv} T. Danielsson, \textit{P. Johnsen}: Homogenization of
the heat equation with a vanishing volumetric heat capacity. Preprint
(2018). arXiv:1809.11019

\bibitem{Jolo2} T. Danielsson, \textit{P. Johnsen}: Homogenization of the
heat equation with a vanishing volumetric heat capacity. Progress in
Industrial Mathematics at ECMI 2018, to appear.

\bibitem{DoWo} \textit{H. Douanla, J. L. Woukeng}: Homogenization of
reaction-diffusion equations in fractured porous media. Electron J.
Differential Equations. \textit{2015} (2015), no. 253, 23 pp. Zbl
1336.35046, MR3414107

\bibitem{FHOPveryweak} \textit{L. Flod\'{e}n, A. Holmbom, M. Olsson, J.
Persson}: Very weak multiscale convergence. Appl. Math. Lett. \textit{23}
(2010), no. 10, 1170--1173. DOI 10.1016/j.aml.2010.05.005, Zbl 1198.35023,
MR2665589

\bibitem{FHOLstrange-term} \textit{L. Flod\'{e}n, A. Holmbom, M. Olsson
Lindberg}: A strange term in the homogenization of parabolic equations with
two spatial and two temporal scales. J. Funct. Spaces Appl. \textit{2012}
(2012), 9 pp. DOI 10.1155/2012/643458, Zbl 1242.35030, MR2875184

\bibitem{FHOLPmismatch} \textit{L. Flod\'{e}n, A. Holmbom, M. Olsson
Lindberg, J. Persson}: A note on parabolic homogenization with a mismatch
between the spatial scales. Abstr. Appl. Anal. \textit{2013}\textbf{\ }%
(2013), 6 pp. DOI 10.1155/2013/329704, Zbl 1293.35027, MR3111807

\bibitem{FHOLParbitrary} \textit{L. Flod\'{e}n, A. Holmbom, M. Olsson
Lindberg, J. Persson}: Homogenization of parabolic equations with an
arbitrary number of scales in both space and time. J. Appl. Math. \textit{%
2014}\textbf{\ }(2014), 16 pp. DOI 10.1155/2014/101685, MR3176810

\bibitem{Hol} \textit{A. Holmbom}: Homogenization of parabolic equations: an
alternative approach and some corrector-type results. Appl. Math. \textit{42}
(1997), no. 5, 321--343. DOI 10.1023/A:1023049608047, Zbl 0898.35008,
MR1467553

\bibitem{JoLo1} \textit{P. Johnsen, T. Lobkova}: Homogenization of a linear
parabolic problem with a certain type of matching between the microscopic
scales. Appl. Math. \textit{63} (2018), no. 5, 503--521. DOI
10.21136/AM.2018.0350-17, Zbl 06986923, MR3870146

\bibitem{Lob} \textit{T. Lobkova}: Homogenization of linear parabolic
equations with a certain resonant matching between rapid spatial and
temporal oscillations in periodically perforated domains. Acta Math. Appl.
Sin. Engl. Ser. \textit{35}\textbf{\ }(2019), no. 2, 340--358. DOI
10.1007/s10255-019-0810-1, MR3950176

\bibitem{LNW} \textit{D. Lukkassen, G. Nguetseng, P. Wall}: Two-scale
convergence. Int. J. Pure Appl. Math. \textit{2} (2002), no. 1, 35--86. Zbl
1061.35015, MR1912819

\bibitem{Ngu1} \textit{G. Nguetseng}: A general convergence result for a
functional related to the theory of homogenization. SIAM J. Math. Anal. 
\textit{20} (1989), no. 3, 608--623. DOI 10.1137/0520043, Zbl 0688.35007,
MR0990867

\bibitem{Ngu2} \textit{G. Nguetseng}: Asymptotic analysis for a stiff
variational problem arising in mechanics. SIAM J. Math. Anal. \textit{21}
(1990), no. 6, 1394--1414. DOI 10.1137/0521078, Zbl 0723.73011, MR1075584

\bibitem{NgWo} \textit{G. Nguetseng, J. L. Woukeng}: $\Sigma $-convergence
of nonlinear parabolic operators. Nonlinear Anal. \textit{66} (2007), no. 4,
968--1004. DOI 10.1016/j.na.2005.12.035, Zbl 1116.35011, MR2288445

\bibitem{PerPhD} \textit{J. Persson}: Selected Topics in Homogenization. Mid
Sweden University Doctoral Thesis 127, Department of Engineering and
Sustainable Development, Mid Sweden University, 2012.

\bibitem{Per} \textit{J. Persson}: Homogenization of monotone parabolic
problems with several temporal scales. Appl. Math. \textit{57} (2012), no.
3, 191--214. DOI 10.1007/s10492-012-0013-z, Zbl 1265.35018, MR2984600

\bibitem{SvWo} \textit{N. Svanstedt, J. L. Woukeng}: Periodic homogenization
of strongly nonlinear reaction-diffusion equations with large reaction
terms. Appl. Anal. \textit{92} (2013), no. 7, 1357--1378. DOI
10.1080/00036811.2012.678334, Zbl 1271.35006, MR3169106

\bibitem{ZeiIIA} \textit{E. Zeidler}: Nonlinear functional analysis and its
applications II/A: linear monotone operators. Springer-Verlag, New York,
1990. DOI 10.1007/978-1-4612-0985-0, Zbl 0684.47028, MR1033497
\end{thebibliography}
\end{document}